\documentclass{scrartcl}

\DeclareOldFontCommand{\rm}{\normalfont\rmfamily}{\mathrm}
\DeclareOldFontCommand{\sf}{\normalfont\sffamily}{\mathsf}
\DeclareOldFontCommand{\tt}{\normalfont\ttfamily}{\mathtt}
\DeclareOldFontCommand{\bf}{\normalfont\bfseries}{\mathbf}
\DeclareOldFontCommand{\it}{\normalfont\itshape}{\mathit}
\DeclareOldFontCommand{\sl}{\normalfont\slshape}{\@nomath\sl}
\DeclareOldFontCommand{\sc}{\normalfont\scshape}{\@nomath\sc}

\usepackage{graphicx}
\usepackage{amsmath} % for eqref
\usepackage{amsthm}

\usepackage{amsfonts} % for mathbb
\usepackage[algoruled, vlined, linesnumbered]{algorithm2e} % for algorithm
\usepackage[noend]{algorithmic} % for algorithm
\usepackage[numbers]{natbib}
\usepackage{multirow}
\usepackage{url}
\usepackage{hyperref}
\usepackage{subfig}

\usepackage{authblk}

\usepackage{mathptmx}     
\usepackage{verbatim}

\newcommand{\bbbr}{\mathbb R}
\newcommand{\mV}{\mathcal{V}}
\newcommand{\mA}{\mathcal{A}}

\newcommand{\mC}{\mathcal{C}}

\newcommand{\mE}{\mathcal{E}}

\newcommand{\mT}{\mathcal{T}}
\newcommand{\mR}{\mathcal{R}}
\newcommand{\mM}{\mathcal{M}}

\newcommand{\tprec}{\text{\emph{prec}}}
\newcommand{\lhull}{{\tt lin.hull}}
\newcommand{\chull}{{\tt conv.hull}}

\providecommand{\keywords}[1]{\textbf{Keywords:} #1}

\newtheorem{theorem}{Theorem}
\newtheorem{lemma}[theorem]{Lemma}
\newtheorem{proposition}[theorem]{Proposition}

\usepackage[usenames, dvipsnames]{color} %Just for comments

\begin{document}

\title{A study of the Bienstock-Zuckerberg algorithm\thanks{The authors of this article would like to acknowledge support by grants Fondecyt 1151098 (MG and ORL), Fondecyt 1130681 (EM), Fondecyt 1150046 (DE), Conicyt PIA Anillo ACT 1407 (DE, MG, EM, and ORL), NSERC RGPIN 5837-08 (MQ) and Conicyt BCH 72130388 (GM).}}

\subtitle{Applications in Mining and Resource Constrained Project Scheduling}

\author[1]{Gonzalo Mu\~noz}
\author[2]{Daniel Espinoza}
\author[3]{Marcos Goycoolea}
\author[4]{Eduardo Moreno}
\author[5]{Maurice Queyranne}
\author[6]{Orlando Rivera}
\affil[1]{Industrial Engineering and Operations Research, Columbia University}
\affil[2]{Gurobi Optimization}
\affil[3]{School of Business, Universidad Adolfo Iba\~nez}
\affil[4]{Faculty of Engineering, Universidad Adolfo Iba\~nez}
\affil[5]{School of Business, University of British Columbia}
\affil[6]{School of Business and Faculty of Engineering, Universidad Adolfo Iba\~nez}

\maketitle

\begin{abstract}
We study a Lagrangian decomposition algorithm recently proposed by Dan Bienstock and Mark Zuckerberg for solving the LP relaxation of a class of open pit mine project scheduling problems. In this study we show that the Bienstock-Zuckerberg (BZ) algorithm can be used to solve LP relaxations corresponding to a much broader class of scheduling problems, including the well-known Resource Constrained Project Scheduling Problem (RCPSP), and multi-modal variants of the RCPSP that consider batch processing of jobs. We present a new, intuitive proof of correctness for the BZ algorithm that works by casting the BZ algorithm as a column generation algorithm. This analysis allows us to draw parallels with the well-known Dantzig-Wolfe decomposition (DW) algorithm. We discuss practical computational techniques for speeding up the performance of the BZ and DW algorithms on project scheduling problems. Finally, we present computational experiments independently testing the effectiveness of the BZ and DW algorithms on different sets of publicly available test instances. Our computational experiments confirm that the BZ algorithm significantly outperforms the DW algorithm for the problems considered. Our computational experiments also show that the proposed speed-up techniques can have a significant impact on the solve time. We provide some insights on what might be explaining this significant difference in performance.\\
\end{abstract}

\section{Introduction}
Resource constrained project scheduling problems (RCPSPs) seek to optimally schedule activities over time in such a way as to comply with precedence and resource usage constraints. These problems can be notoriously difficult. Despite great progress in optimization methodologies during the last fifty years, there are instances of RCPSP involving as few as sixty activities that cannot be solved with today's most effective algorithms \cite{psplib}. In multi-modal extensions of RCPSP, jobs can be processed in different ways (or modes). Changing the mode of a job affects its resource utilization, and its processing costs. In some applications, changing the mode of a job changes the duration of its processing time. In batch-processing extensions of RCPSP jobs are grouped together in clusters that must be processed together. For simplicity of notation, we will henceforth refer to multi-modal batch resource constrained project scheduling problems, simply as General Production Scheduling Problems, or GPSPs.

To date, the most effective methods for solving GPSPs, especially modal and batch variants, are based on integer programming methods that use the so-called time index formulations \cite{zhu06, berthold10, boland09, chicoisne12}. These formulations define a binary variable for each job-execution time combination. An important limitation of these formulations is that the linear programming (LP) relaxations are very difficult to solve. This is because they tend to be large and degenerate, even for small problem instances. 

While classical decomposition techniques and column generation approaches can be used for addressing some of these issues (specially complications related to the large number of variables), they often suffer from slow convergence rate. In this context, an important contribution was made by Bienstock and Zuckerberg \cite{bienstock09,bienstock10}, where the authors presented an alternative to tackle these limitations effectively. They developed a new algorithm that can considerably outperform classical methods in a broad class of problems, thus providing a novel set of tools with a high practical impact and a wide range of potential extensions.

In this paper we study the Bienstock-Zuckerberg (BZ) algorithm \cite{bienstock09,bienstock10} in depth, and find that it can be used to overcome the stated limitations on a wide range of scheduling problems. We provide additional evidence to that in  \cite{bienstock09,bienstock10}, advocating for the efficacy of the algorithm in practice, and we provide new insights on it that allow further extensions. Specifically, we study the BZ algorithm as an approach to batch, multi-modal production scheduling problems where jobs can have arbitrary durations, but where these durations are not affected by changes of mode. The application of the BZ algorithm to this class of problems requires us to extend the algorithmic template as it was originally proposed. We are specifically concerned about this class of problems because, in addition to generalizing the well-known RCPSP, it generalizes three very important problems that arise in the context of mine planning: Underground Production Scheduling Problems (UPSPs), Open Pit Phase Design Problems (OPPDPs), and Open Pit Production Scheduling Problems (OPPSPs). 

We present a new proof of algorithm correctness by casting the BZ algorithm as a column generation method, discuss algorithmic speed-ups, computationally test it on a variety of problem instances, and compare the methodology to the more traditional Dantzig-Wolfe decomposition (DW) method. As part of our analysis we prove that the BZ algorithm is closely related to a decomposition scheme that produces a bound somewhere in between that of the DW method, and that of the LP relaxation. This study allows us to conclude that the BZ algorithm is an effective way of solving the LP relaxation of large time index formulations, significantly outperforming the DW method on all considered problem classes, and provides insights on the effectiveness of the methodology.

This article is organized as follows. In Section \ref{sec:background} we present a literature review of related scheduling work in mine planning applications and mathematical programming methodologies. In Section \ref{sec:formulation} we present an integer programming model that generalizes the classes of problems we are interested in studying. In Section \ref{sec:reformulation} we present a reformulation of this model that fits the algorithmic framework we will be analyzing. In Section \ref{sec:methodology} we present a general column generation framework that can be used to solve the problem. We use this column generation approach to motivate the BZ algorithm, and facilitate a comparison to the DW method. We also introduce important speed ups for the BZ algorithm. Computational results analyzing the performance of BZ and comparing this performance to that of DW are presented in Section \ref{sec:computations}.

\section{Background}\label{sec:background}

{\bf Scheduling applications in Mine Planning}~\\

In this article we are mainly interested in addressing scheduling problems that are of relevance to planning applications in the mining industry.  

The class of mining problems we are interested in include open pit and underground mine planning problems. In these problems, deposits are discretized into three-dimensional arrays known as block models. The problem to solve consists of deciding which blocks should be extracted, when they should be extracted, and what to do with the blocks once they are extracted. In this context, jobs correspond to extraction activities, modes correspond to different processing options and batches correspond to groups of blocks that must be extracted concurrently in order to comply with equipment extraction requirements. Resource constraints would be used to impose extracting, milling and refining capacity constraints. In an open-pit mine, precedence constraints would be used to impose that a mine must be extracted from the surface on downwards. In an under-ground mine, specifically in a stopping operation, precedence constraints would be used to impose that selected blocks branch out from a network of tunnels descending from the surface. 

In all of these mining problems the objective is to maximize net-present-value of the extracted minerals. This is different from traditional scheduling problems, in which the objective is typically to minimize completion time metrics such as makespan, maximum tardiness, or total throughput time. Another difference between mine planning problems and traditional scheduling problem is that in mining it is not a strict requirement to execute all jobs. In all other ways, the Underground Production Scheduling Problems (UPSPs) is identical to the RCPSP. The Open Pit Phase Design Problems (OPPDPs) is a multi-modal RCPSP in which all jobs take a single-time period to complete, regardless of the selected mode. What the mode affects is the objective function value and the resource consumption of each activity. The Open Pit Production Scheduling Problems (OPPSPs) is just like the OPPDP, with the addition that there are batch constraints that force groups of blocks to be extracted simultaneously. These batch constraints are used to enforce equipment operability constraints, with each batch corresponding to contiguous set of blocks typically called a \emph{bench-phase} or \emph{increment} in the mining industry. 

For an introduction to optimization in underground mining see Alford et al. \cite{alford07}. For related work in underground mining see Martinez and Newman \cite{martinez11}, Newman and Kuchta \cite{newman07} and O'Sullivan et al. \cite{osullivan14,osullivan15}. The user manuals of Deswik Scheduler \cite{deswik-scheduler} and MineMax iGantt \cite{minemax-igantt} illustrate how UPSP is solved in practical mining applications. For an introduction to Open Pit Mining see Hustrulid and Kuchta \cite{hustrulid06}, and the user manuals of Dassault Whittle \cite{whittle15}, MineMax Scheduler \cite{minemax-scheduler}, and MineMax Planner \cite{minemax-planner}. For a discussion on OPPSP see Goycoolea et al. \cite{goycoolea15}. For a general survey on OR applications in mine planning see Newman et al. \cite{newman10}.\\
 
\noindent {\bf Mathematical programming methodologies}~\\

To our knowledge, the first mathematical programming formulations of GPSPs dates back to Johnson \cite{johnson68} and Pritsker et al. \cite{pritsker69}, in the late 1960s. Each of these articles spawned its own track of academic articles on production scheduling problems. The first track, following the work of Johnson, mostly focused on strategic open pit mine planning problems. The second track, following the work of Pritsker et al., took a more general approach. Surprisingly, though many of the results found in these two tracks coincide, there are few, if any, citations connecting the two literatures.

The academic literature on exact optimization methods for GPSPs is immensely rich. Well-known surveys from the scheduling track include those of Graham et al. \cite{graham79}, Brucker et al. \cite{brucker99} and Hartmann and Briskorn \cite{hartmann10}. In a recent book by Artigues et al. \cite{artigues10} a very thorough survey of GPSP is presented, including a computational study that compares existing methodologies on benchmark instances. Surveys from the mining track include those of Newman et al. \cite{newman10}, Espinoza et al. \cite{espinoza13b} and Osanloo et al. \cite{osanloo08}.  

A brief summary of advances on solving the LP relaxation of time-index formulations is as follows. Since as early as the 1960s, most methods have been based on some form of decomposition that reduces the problem to a sequence of maximum closure or minimum cut problems. Johnson \cite{johnson68}, in the context of mining (single mode OPPDPs), was the first to attempt such an approach with a Dantzig-Wolfe decomposition. Shortly after, and independently, Fisher \cite{fisher73}, in the context of scheduling, proposed a Lagrangian Relaxation decomposition approach. Since then, a number of decomposition algorithms, primarily Lagrangian Relaxation decomposition algorithms, have been proposed for the problem in both literatures. Important examples include Dagdelen and Johnson \cite{dagdelen86} and Lambert and Newman \cite{lambert14} in mining (single mode OPPDPs), and Christofides et al. \cite{christofides87} and Mohring et al. \cite{mohring03} in scheduling. 

Some recent approaches are as follows. Chicoisne et al. \cite{chicoisne12} consider single-mode OPPDPs with a single renewable resource, and propose a decomposition algorithm that solves the continuous relaxation in polynomial time. Boland et al. \cite{boland09} consider a multi-modal variant of the same problem with two renewable resources, and propose a decomposition algorithm for the continuous relaxation that, while not provably polynomial, is very effective in practice. The BZ algorithm \cite{bienstock09}, developed shortly after, can be considered a generalization of this last algorithm that extends to multi-modal OPPMPs with an arbitrary number of renewable or non-renewable resources. This algorithm proved to be very efficient, even for 
extremely large instance sizes; it reduced the solving times drastically and it was able to tackle instances that could not be solved before. Berthold et al. \cite{berthold10} developed a branch-and-bound algorithm that combines mathematical programming and constraint programming methods to close a large number of open benchmark instances of RCPSP. Zhu et al. \cite{zhu06} developed a mathematical programming based algorithm for solving multi-modal variants of RCPSP without batch constraints, but where mode changes affect duration times. In his work he closes a large percentage of open benchmark instances. 

\section{Integer programming formulation}\label{sec:formulation}

We now describe an integer programming formulation for a class of production scheduling problems that generalizes the RCPSP, UPSP, OPPDP and OPPSP. As mentioned before, we simply refer to this class of problems as the General Production Scheduling Problem (GPSP). 

\noindent {\textbf{Sets:}}
\begin{itemize}
\item $\mA$ : activities that must be scheduled in the problem.
\item $\mC$ : clusters of activities that define a partition of $\mA$.
\item $\tprec(c)$ : clusters that must be initiated no later than cluster $c \in \mC$.
\item $\mR = \{1,\ldots,R\}$ : resources that are consumed when carrying out activities. 
\item $\mT = \{1,\ldots,T\}$ : time periods in which it is possible to initiate activities. 
\item $\mM_a = \{1,\ldots, M_a\}$ : possible modes for activity $a \in \mA$.
\end{itemize}

\noindent {\textbf{Parameters:}}
\begin{itemize}
\item $t^-_a$ : release date for activity $a \in \mA$.
\item $t^+_a$ : due date for activity $a \in \mA$.
\item $d_{a,m}$ : duration (number of time periods) of activity $a \in \mA$ when executed in mode $m \in \mM_a$.
\item $p_{a,m,t}$ : profit obtained if activity $a \in \mA$ is initiated in period $t \in \mT$ using mode $m \in \mM_a$.
\item $l(c_1,c_2)$ : lag (number of time periods) that must elapse before activities in cluster $c_2 \in \mC$ can be initiated, after activities in cluster $c_1 \in \tprec(c_2)$ have been initiated. 
\item $Q_{r,t}$ : amount of resource $r \in \mR$ available in time period $t \in \mT$. 
\item $q_{r,a,m}$ : amount of resource $r \in \mR$ consumed by activity $a \in \mA$ in each time period it is being executed, when executed in mode $m \in \mM_a$.
\end{itemize}

\noindent {\textbf{Variables:}}
\begin{align*}
x_{c,t}  & = \left\{ \begin{array}{rl} 1 & \text{ if the activities in cluster } c  \text{ all start in time period } t \\ 0 & \text{ otherwise} \end{array}\right.\\
y_{a,m,t} & = \left\{ \begin{array}{rl} 1 & \text{ if activity } a \text{ starts in time period } t \text{ using mode } m \\ 0 & \text{ otherwise} \end{array}\right.\\
\end{align*}

\noindent {\textbf{Objective function:}}
\begin{equation}\label{eq:psp:obj}
\text{maximize } \sum_{a \in \mA} \sum_{m \in \mM_a} \sum_{t \in \mT} p_{a,m,t} y_{a,m,t}
\end{equation}
Note that we express the objective function only in terms of the $y$ variables. There is no loss of generality in this due to constraints \eqref{eq:psp:ac}, below.\\

\noindent {\textbf{Constraints:}}~\\

\noindent Clusters can only be initiated once over the time horizon [ $\forall c \in \mC$ ]:
\begin{equation}\label{eq:psp:once}
\sum\limits_{t \in \mT} x_{c,t} \leq 1.
\end{equation}
Activities in a cluster must start simultaneously, and must be carried out using a unique mode [ $\forall c \in \mC, \; \forall a \in c, \; \forall t \in \mT$ ]:
\begin{equation}\label{eq:psp:ac}
x_{c,t} = \sum\limits_{m \in \mM_a} y_{a,m,t}.
\end{equation}
In order to initiate the activities in a cluster, all activities in preceding clusters must be initiated early enough so as to satisfy the lag requirement \;\; [ $ \forall c_2 \in \mC, \; \forall c_1  \in \tprec(c_2), \; \forall t \in \mT$ ]:
\begin{equation}\label{eq:psp:prec}
\sum\limits_{s \leq t} x_{c_2,s} \leq \sum\limits_{s \leq t - l(c_1,c_2) } x_{c_1,s}.
\end{equation}
The amount of resources consumed cannot exceed the amount of resources available each period [ $\forall r \in \mR, \; t \in \mT$ ]:
\begin{equation}\label{eq:psp:resources}
\sum\limits_{a \in \mA} \sum\limits_{m \in \mM_a}  q_{r,a,m} \sum\limits_{s=\max\{1,t-d_{a,m}+1\}}^{t} y_{a,m,s} \leq Q_{r,t}.
\end{equation}

Activities can only be scheduled after their release dates [ $\forall a \in \mA$ ]:
\begin{equation}\label{eq:psp:release}
\sum\limits_{m \in \mM_a} \sum\limits_{t = 1}^{t^-_a - 1} y_{a,m,t} = 0.
\end{equation}
Activities must be terminated no later than due dates and exactly one mode must be chosen for each executed activity [ $\forall a \in \mA : t^+_a < \infty$ ]:
\begin{equation}\label{eq:psp:due}
\sum\limits_{m \in \mM_a} \sum\limits_{t=1}^{t^+_a - d_{a,m} + 1} y_{a,m,t} = 1.
\end{equation}

Whenever $t^+_a = \infty$, there is no due date for activity $a$ and we do not include constraint \eqref{eq:psp:due}. Note, however, that due to \eqref{eq:psp:once} and \eqref{eq:psp:ac} we always have:
$$   \sum\limits_{m \in \mM_a} \sum\limits_{t \in \mT} y_{a,m,t} \leq 1.$$
Thus, even when $t^+_a = \infty$, there is an implicit constraint enforcing the choice of one mode and one time period at most.

It should be noted that the GPSP described by Formulation \eqref{eq:psp:obj}-\eqref{eq:psp:due} generalizes the scheduling formulations found in Bienstock and Zuckerberg \cite{bienstock10} in that it allows activities to have durations that span multiple time periods. This allows us to consider instances of RCPSP and UPSP which fell outside the scope of the Bienstock and Zuckerberg studies \cite{bienstock09, bienstock10}.

Though this formulation is intuitive, and also the most commonly used formulation in the academic literature, it must be reformulated so as to fit the algorithmic schemes that will be presented.

\section{Reformulation}\label{sec:reformulation}

In this section we present a reformulation of the GPSP formulation described in Section \ref{sec:formulation}. As we will see in Section \ref{sec:methodology}, this reformulation is key for the decomposition algorithms that we will discuss in this article.\\

\noindent For each $c \in \mC, a \in \mA, m \in \mM_a$ and $t \in \mT$ define,
\[ w_{c,t} = \sum\limits_{s=1}^t x_{c,s} \qquad \text{ and } \qquad z_{a,m,t} = \sum\limits_{i \in \mM_a} \sum\limits_{s=1}^{t-1} y_{a,i,s} + \sum\limits_{i=1}^m y_{a,i,t}. \]

\noindent In this way, $w_{c,t}$ is a binary variable that takes value one if and only if cluster $c$ is initiated ``by'' time $t$ (i.e., no later than $t$). Likewise, $z_{a,m,t}$ is a binary variable that takes value one if and only if activity $a$ is initiated by time $t-1$, or, if it is initiated in time period $t$, and its mode $i$ is such that $i \leq m$.\\

\noindent To see that it is possible to formulate the production scheduling problem given by \eqref{eq:psp:obj}-\eqref{eq:psp:due}, using the $(z,w)$ variables, note that we can map between the two variable spaces by means of the following linear sets of equations:
\begin{subequations}
\begin{align}
y_{a,m,t} & = z_{a,m,t} - z_{a,m-1,t} & \qquad \forall a \in \mA, \; m = 2, \ldots, M_a, \; t \in \mT \\
y_{a,1,t} & = z_{a,1,t} - z_{a,M_a,t-1} & \qquad \forall a \in \mA, \; t = 2, \ldots, T \\
y_{a,1,1} & = z_{a,1,1} & \qquad \forall a \in \mA \\
x_{c,t} & = w_{c,t} - w_{c,t-1} & \qquad \forall c \in \mC, \; \forall t = 2, \ldots, T \\
x_{c,1} & = w_{c,1} & \qquad \forall c \in \mC
\end{align}
\end{subequations}
Using this mapping we can substitute out the variables in \eqref{eq:psp:obj}-\eqref{eq:psp:due}, to trivially obtain the following equivalent formulation:\\

\noindent {\textbf{Objective function:}}
\begin{equation}\label{psp2:obj}
\text{max } \sum_{a \in \mA} \sum_{m \in \mM_a} \sum_{t \in \mT} \tilde p_{a,m,t} z_{a,m,t}
\end{equation}
\indent where,
\[ \tilde{p}_{a,m,t} = \left\{ 
\begin{array}{ll} 
p_{a,m,t} - p_{a,m+1,t} & \text{ if } m < M_a \\
p_{a,m,t} - p_{a,1,t+1} & \text{ if } m = M_a, \; t < T \\
p_{a,m,t}               & \text{ if } m = M_a, \; t = T \end{array} \right.
\]

\noindent {\textbf{Constraints:}}~\\

\noindent For $a \in \mA, \; m \in \{1,\ldots,M_a-1\}, \; t \in \mT$,:
\begin{equation}\label{psp2:zprec1}
z_{a,m,t} \leq  z_{a,m+1,t}.
\end{equation}
\noindent For $a \in \mA, \; m = M_a, \; t \in \{1,\ldots,T-1\}$,:
\begin{equation}\label{psp2:zprec2}
z_{a,m,t} \leq  z_{a,1,t+1}.
\end{equation}
For $c \in \mC, \; a \in c, \; t \in \mT$,
\begin{equation}\label{eq:wz}
w_{c,t} = z_{a,M_a,t}.
\end{equation}
For $c_2 \in \mC, \; c_1 \in \tprec(c_2), \; t \in \{1+l(c_1,c_2),\ldots,T\}$,
\begin{equation}\label{psp2:wprec}
w_{c_2,t} \leq  w_{c_1,t-l(c_1,c_2)}.
\end{equation}
For $r \in R, \; t \in \mT$,
\begin{equation}\label{psp2:resources}
\sum\limits_{a \in \mA} \sum\limits_{m \in \mM_a} \sum\limits_{s\in \mT } \tilde{q}^{r,t}_{a,m,s} z_{a,m,s} \leq Q_{r,t}.
\end{equation}
\indent where \[ \tilde{q}^{r,t}_{a,m,s} = \left\{ 
\begin{array}{ll} 
q_{r,a,m} \mathbf{1}_{[t-d_{a,m}+1,t]}(s)  - q_{r,a,m+1}  \mathbf{1}_{[t-d_{a,m+1}+1,t]}(s) & \text{ if }  m < M_a \\
q_{r,a,M_a} \mathbf{1}_{[t-d_{a,M_a}+1,t]}(s) - q_{r,a,1} \mathbf{1}_{[t-d_{a,1},t-1]}(s) & \text{ if } m = M_a \end{array} \right.\]
\indent and,
\[ \mathbf{1}_{[x,y]}(s) = \left\{\begin{array}{ll} 1 & \text{ if } x \leq s \leq y \\ 0 & \text{ otherwise.} \end{array}\right. \]
(The derivation of this last constraint from \eqref{eq:psp:resources} can be found in Appendix \ref{sec:resources}).\\
For $a \in \mA, \; m \in \mM_a, \; t < t^-_a$:
\begin{equation}\label{psp2:zfixed1}
z_{a,m,t} = 0.
\end{equation}
For $a \in \mA, \; m \in \mM_a, \; t \geq t^+_a$:
\begin{equation}\label{psp2:zfixed2}
z_{a,m,t} = 1.
\end{equation}

After the reformulation, if we substitute out the $w$ variables by using linear equalities \eqref{eq:wz}, we obtain what Bienstock and Zuckerberg \cite{bienstock10} call a General Precedence Constrained Problem (GPCP) :

\begin{eqnarray}
Z^* = & \max \ & \quad c'z \\
& \mbox{s.t.} \ & z_i \leq z_j \qquad \forall (i,j) \in I, \label{psp3:prec} \\
&& Hz \leq h, \label{psp3:side} \\
&& z \in \{0,1\}^n 
\end{eqnarray}

\noindent In this problem constraints \eqref{psp3:prec} correspond to constraints \eqref{psp2:zprec1}, \eqref{psp2:zprec2}, and \eqref{psp2:wprec}, and constraints \eqref{psp3:side} correspond to constraints \eqref{psp2:resources}, \eqref{psp2:zfixed1} and \eqref{psp2:zfixed2}.

It is interesting that despite the fact that the General Production Scheduling Problem (GPSP) formulated in Section \ref{sec:formulation} is more general than the scheduling problem formulations presented by Bienstock and Zuckerberg \cite{bienstock09}, both problems can be reformulated as an instance of GPCP. 

\section{Methodology.}\label{sec:methodology}

In this section we describe a generalized version of the Bienstock-Zuckerberg (BZ) algorithm, originally introduced in \cite{bienstock09, bienstock10}. More specifically, we describe a decomposition algorithm well suited for solving the LP relaxation of large mixed integer programming problems having form,

\begin{equation}\label{prob:gpcp}
\begin{array}{rl}
Z^{IP} = \max \ \; & c'z + d'u \\
\mbox{s.t.} \ \; & z_i \leq z_j, \qquad \forall (i,j) \in I, \\
& Hz + Gu \leq h,  \\
& z \in \{0,1\}^n.
\end{array}
\end{equation}

\noindent Like Bienstock and Zuckerberg \cite{bienstock10} before us, we will call this problem the General Precedence Constrained Problem (GPCP). It should be noted, however, that in our definition of GPCP we consider the presence of the extra $u$ variables. These variables can be used for other modeling purposes. For example, they can be used to model variable capacities, or, in the context of mining problems, they can be used to model the use of stockpiles or other requirements. 

Our presentation of the BZ algorithm is different than the original presentation in two respects: first, we cast it as a column generation method, and second, as stated before, we consider the presence of extra variables (the $u$ variables in \eqref{prob:gpcp}). These differences require a new proof of algorithm correctness that we present below. The advantage of presenting the algorithm this way is that it is easier to compare to existing decomposition algorithms, and thus, in our opinion, becomes easier to understand and extend it. It also opens up the possibility of developing new classes of algorithms that might be useful in other contexts.

Before we present the BZ algorithm, we present a generic column generation (GCG) algorithm that can be used as a framework for understanding the BZ algorithm. This column generation scheme can be used to solve a relaxation of mixed integer problems having form: 

\begin{equation}\label{prob:gen-ip}
\begin{array}{rl}
Z^{IP} = \max \ \; & c'z + d'u \\
\mbox{s.t.} \ \; & Az \leq b, \\
& Hz + Gu \leq h,  \\
& z \in \{0,1\}^n.
\end{array}
\end{equation}

Much like DW algorithm, this relaxation can yield bounds that are tighter than those obtained by solving the LP relaxation of the same problem. Throughout this section we will assume, without loss of generality, that $Az \leq b$ includes $0 \leq z \leq 1$. Other than this, we make no additional assumption on problem structure. 

\subsection{A general column generation algorithm}\label{sec:cg}

In this section we introduce a General Column Generation (GCG) algorithm that will later motivate the BZ algorithm. This column generation algorithm is presented as a decomposition scheme for computing upper bounds of \eqref{prob:gen-ip} that are no weaker than those provided by the LP relaxation. 

Given a set $S \subseteq R^n$, let \lhull$(S)$ denote the linear space spanned by $S$. That is, let \lhull($S$) represent the smallest linear space containing $S$. Let $P = \{ z \in \{0,1\}^n : Az \leq b \}$. Define problem,

\begin{equation}\label{prob:lin}
\begin{array}{rl}
Z^{LIN} = \max \ \; & c'z + d'u \\
\mbox{s.t.} \ \;& Az \leq b, \\
& Hz + Gu \leq h, \\
& z \in \textrm{\lhull}(P),
\end{array}
\end{equation}

The GCG algorithm that we present computes a value $Z^{UB}$ such that $Z^{IP} \leq Z^{UB} \leq Z^{LIN}$. Observe that the optimal value $Z^{LIN}$ of problem \eqref{prob:lin} is such that $Z^{IP} \leq Z^{LIN} \leq Z^{LP}$, where $Z^{LP}$ is the value of the linear relaxation of problem \eqref{prob:gen-ip}. 

The key to solving this problem is the observation that
\begin{equation}\label{eq:lhull}
\textrm{\lhull}(P) = \{ z : z = \sum\limits_{i = 1}^d \lambda_i v^i, \; \textrm{ for some } \lambda \in \mathbb{R}^n \},
\end{equation}
where $\{v^1,\ldots,v^d\}$ is a basis for \lhull$(P)$. If $V$ is a matrix whose columns $\{v^1,\ldots,v^d\}$ define a basis of \lhull$(P)$, it follows that problem \eqref{prob:lin} is equivalent to solving,

\begin{equation}\label{prob:gen-lb}
\begin{array}{llll}
Z(V) = & \max \ & c'V \lambda + d'u & \\
& \mbox{s.t.} \ & A V \lambda \leq b & (\alpha) \\
&& H V \lambda + G u \leq h & (\pi).
\end{array}
\end{equation}

In order for the optimal solution of \eqref{prob:gen-lb} to be optimal for  \eqref{prob:lin}, it is not necessary for the columns of $V$ to to define a basis of \lhull$(P$). It suffices for the columns of $V$ to span a subset of \lhull$(P)$ containing an optimal solution of \eqref{prob:lin}. This weaker condition is what the column generation seeks to achieve. That is, starting with a matrix $V$ with columns spanning at least one feasible solution of \eqref{prob:lin}, the algorithm iteratively computes new columns until it computes the optimal value $Z^{LIN}$. As we will see, in some cases it is possible for the algorithm to terminate before having computed $Z^{LIN}$. In these cases the algorithm will terminate having computed a value $Z^{UB}$ such that $Z^{UB} \leq Z^{LIN}$.

Henceforth, let $\alpha, \pi$ denote the dual variables corresponding to the constraints of problem \eqref{prob:gen-lb}. To simplify notation, we will henceforth assume that $\alpha$ and $\pi$ are always row vectors. Note that $\alpha, \pi \geq 0$.

To solve problem \eqref{prob:lin} with column generation we begin each iteration with a set of points $\{v^1, \ldots, v^k\}$ such that \lhull$(\{v^1, \ldots, v^k\}) \subseteq $\lhull$(P)$ contains at least one feasible solution of problem \eqref{prob:lin}. At each iteration we add a linearly independent point $v^{k+1}$ to this set, until we can prove that we have generated enough points so as to generate the optimal solution of \eqref{prob:lin}.

The first step of each iteration consists in solving problem \eqref{prob:gen-lb}, with a matrix $V^k$ having columns $\{v^1, \ldots, v^k\}$. Let $Z^L_k = Z(V^k)$. Let ($\lambda^k, u^k)$ and $(\alpha^k, \pi^k)$ be the corresponding optimal primal and dual solutions. Note that $Z^L_k = \alpha^k b + \pi^k h$.

Define $z^k = V^k \lambda^k$. It is clear that $(z^k, u^k)$ is feasible for \eqref{prob:lin}; hence, $Z^L_k \leq Z^{LIN}$. However, is $(z^k, u^k)$ optimal for \eqref{prob:lin}? To answer this question, observe that $(\alpha,\pi)$ is dual-feasible for problem \eqref{prob:gen-lb} if and only if $\pi G = d'$ and,

\[ \bar c(v, \alpha, \pi) \equiv c' v - \alpha Av - \pi H v = 0, \qquad \forall v \in V. \]

Let $V^P$ represent a matrix whose columns $\{v^1,\ldots,v^{\small |P|}\}$ coincide with set $P$. Since $V^P$ is clearly a generator of \lhull$(P)$, we know that the optimal solution of \eqref{prob:gen-lb} (for $V^P$) corresponds to an optimal solution of \eqref{prob:lin}. Thus, if $\bar c(v, \alpha^k, \pi^k) = 0$ for all $v \in V^P$ (or equivalently, $v \in P$), we conclude that $(z^k, u^k)$ is optimal for \eqref{prob:lin}, since we already know that $\pi^k G = d'$. On the other hand, if $(z^k, u^k)$ is not optional for \eqref{prob:lin}, we conclude that there must exist $v \in P$ such that $\bar c(v, \alpha^k, \pi^k) \neq 0$. 

This suggests how the column generation scheme for computing an optimal solution of \eqref{prob:lin} should work. Solve \eqref{prob:gen-lb} with $V^k$ to obtain primal and dual solutions $(z^k, u^k)$ and $(\alpha^k, \pi^k)$. Keeping with traditional column generation schemes, we refer to problem \eqref{prob:gen-lb} as the \emph{restricted master problem}. If $\bar c(v, \alpha^k, \pi^k) = 0$ for all $v \in P$, then $(z^k,u^k)$ is an optimal solution of \eqref{prob:lin}. If this condition does not hold, let $v^{k+1} \in P$ be such that $\bar c(v^{k+1}, \alpha^k, \pi^k) \not= 0$.  Note that $v^{k+1}$ must be linearly independent of the columns in $V^k$. In fact, every column $v$ of $V^k$ must satisfy $\bar{c} (v, \alpha^k, \pi^k) = 0$, hence so must any vector that can be generated with this set. We refer to the problem of computing a vector $v^{k+1}$ with $\bar{c}(v^{k+1}, \alpha^k, \pi^k) \not= 0$ as the \emph{pricing problem}. Obtain a new matrix $V^{k+1}$ by adding column $v^{k+1}$ to $V^k$, and repeat the process. Given that the rank of $V^k$ increases strictly with each iteration, the algorithm must finitely terminate.

The pricing problem can be tackled by solving:

\begin{equation}\label{prob:gen-ub}
\begin{array}{llll}
L(\pi) = & \max \ & c' v - \pi(Hv - h) & \\
& \mbox{s.t.} \ & Av \leq b, & \\
&& v \in \{0,1\}^n &
\end{array}
\end{equation}

For this, at each iteration compute $L(\pi^k)$, and let $\hat v$ be the corresponding optimal solution. Observe that $L(\pi^k) \geq Z^{IP}$ for all $k \geq 1$. In fact, since the $u$ variables are free in problem \eqref{prob:gen-lb}, and since $(\alpha^k,\pi^k)$ is dual-feasible in problem \eqref{prob:gen-lb}, we have that $d' - \pi^k G = 0$. This implies that problem \eqref{prob:gen-ub}, for $\pi^k$, is equivalent to
\begin{equation}\label{prob:gen-ub-2}
\begin{array}{llll}
& \max \ & c' v + d' u - \pi^k(Hv + G u - h) & \\
& \mbox{s.t.} \ & Av \leq b, & \\
&& v \in \{0,1\}^n. &
\end{array}
\end{equation}
Given that $\pi^k \geq 0$, problem \eqref{prob:gen-ub-2} is a relaxation of problem \eqref{prob:gen-ip}, obtained by penalizing constraints $Hz + Gu \leq h$. Hence $L(\pi^k) \geq Z^{IP}$ $\forall k\ \geq 1$.

Define $Z^U_k = \min \{ L(\pi^i) : i = 1,\ldots, k \}$ and note that $Z^U_k \geq Z^{IP}$, by the argument above. If $Z^U_k \leq Z^L_k$, we can define $Z^{UB} = Z^U_k$, and terminate the algorithm, as we will have that $Z^{IP} \leq Z^{UB} \leq Z^{LIN}$. In fact, since $Z_k^L \leq Z^{LIN}$, we have $Z^{IP} \leq Z^{UB} = Z^U_k \leq Z^L_k \leq Z^{LIN}$.

On the other hand, if $Z^U_k > Z^L_k$, then the optimal solution $\hat v$ of \eqref{prob:gen-ub} is such that $\bar c(\hat v, \alpha^k, \pi^k) \neq 0$, so we obtain an entering column by defining $v^{k+1} = \hat v$. To see this, note that 
\begin{align*}
Z^U_k - Z^L_k > 0 & \Leftrightarrow c'\hat v - \pi^k (H\hat v - h) - \alpha^k b - \pi^k h > 0 \\           
& \Leftrightarrow c'\hat v - \pi^k H\hat v - \alpha^k b > 0 \\           
& \Rightarrow c' \hat v - \pi^k H \hat v - \alpha^k A \hat v > 0 \quad (\textrm{since } \alpha^k b \geq \alpha^k A \hat v ) \\
& \Rightarrow \bar c(\hat v, \alpha^k, \pi^k) > 0. 
\end{align*} 

Observe that, to be precise, we do not need to know a starting feasible solution in order to solve problem \eqref{prob:lin}. In fact, if we do not have a feasible solution, we can add a variable to the right-hand-side of each constraint $Hz + Gu \leq h$, and heavily penalize it to proceed as in a Phase-I simplex algorithm (see \cite{bertsimas97}). The first iteration only needs a solution $v^1$ such that $Av^1 \leq b$.

\subsection{Comparison to the Dantzig-Wolfe decomposition algorithm.}\label{sec:dw}

When solving problems with form \eqref{prob:gen-ip} a common technique is to use the Dantzig-Wolfe decomposition \cite{dantzig60} to compute relaxation values. The Dantzig-Wolfe decomposition (DW) algorithm is very similar to the General Column Generation (GCG) algorithm presented in the previous section. In fact, the DW algorithm computes the optimal value of problem,
\begin{equation}\label{prob:dw}
\begin{array}{rl}
Z^{DW} = \max \ & c'z + d'u \\
\mbox{s.t.} \ & Az \leq b, \\
& Hz + Gu \leq h, \\
& z \in \textrm{\chull}(P).
\end{array}
\end{equation}

\noindent Given that \chull$(P) \subseteq \{ z : Az \leq b \}$, this problem is typically written as,
\begin{equation}\label{prob:dw2}
\begin{array}{rl}
Z^{DW} = \max \ & c'z + d'u \\
\mbox{s.t.} \ & z \in \textrm{\chull}(P), \\
& Hz + Gu \leq h, 
\end{array}
\end{equation}

or equivalently, as 
\begin{equation}\label{prob:dw3}
\begin{array}{rl}
Z^{DW} = \max \ & c'V \lambda + d'u \\
\mbox{s.t.} \ & \lambda \cdot 1 = 1, \\
& H V \lambda + Gu \leq h, \\
& \lambda \geq 0.
\end{array}
\end{equation}

In this formulation, the columns $\{v^1,\ldots,v^k\}$ of $V$ correspond to the extreme points of \chull$(P)$.

Given that \chull$(P) \subseteq$ \lhull$(P)$, it follows that $Z^{IP} \leq Z^{DW} \leq Z^{LIN} \leq Z^{LP}$. When $\chull(P) = \{ z : Az \leq b \}$ it is easy to see that $Z^{DW} = Z^{LIN} = Z^{LP}$. However, the following example shows that all of these inequalities can also be strict:
\begin{equation}\label{prob:example}
\begin{array}{rl}
Z^{IP} = \max \ & -x_1 + 2x_2 + x_3 \\
\mbox{s.t.} \ & -x_1 + x_2 \leq 0.5, \\
& x_1 + x_2 \leq 1.5, \\
& x_3 \leq 0.5, \\
& x \in \{0,1\}^3.
\end{array}
\end{equation}

\noindent By decomposing such that $Hz + Gu \leq h$ corresponds to $-x_1 + x_2 \leq 0.5$ we get,
\[ \{ z : Az \leq b \} = \{ z \in [0,1]^3 :  x_1 + x_2 \leq 1.5, \; x_3 \leq 0.5 \}, \]
and,
\[ P =  \{ z \in \{0,1\}^n  : Az \leq b \} = \{ (0,0,0), (0,1,0),(1,0,0) \}, \]
and so,
\begin{align*}
 \textrm{\lhull}(P) & = \{ (x_1,x_2,x_3) : x_3 = 0 \}, \\
 \textrm{\chull}(P) & = \{ (x_1,x_2,x_3) : 0 \leq x_1, \;0\leq x_2, \; x_1+x_2 \leq 1, \; x_3 = 0 \}.
\end{align*}
From this it can be verified that $z^{IP} = 0.0, \; z^{DW} = 1.25, \; z^{LIN} = 1.5, $ and $z^{LP} = 2.0$.

The DW algorithm is formally presented in Algorithm \ref{alg:dw}. The proof of correctness is strictly analogous to the proof presented in Section \ref{sec:cg} for the generic column generation algorithm.

\begin{figure}[ht!]
\begin{algorithm}[H]
\caption{The DW Algorithm.}\label{alg:dw}
\KwIn{A feasible mixed integer programming problem of form
\begin{equation}\label{alg:prob:gendw}
\begin{array}{llll}
Z^* = & \max \ & c'z + d'u & \\
& \mbox{s.t.} \ & Az \leq b & \\
&& Hz + Gu \leq h & \\
&& z \in \{0,1\}. 
\end{array}
\end{equation}
A matrix $V$ whose columns are feasible $0$-$1$ solutions of $Az \leq b$.\\}
\KwOut{
An optimal solution $(z^*,u^*)$ to
problem
\begin{equation}\label{alg:prob:gen:dw}
\begin{array}{rl}
Z^{DW} = \max \ & c'z + d'u \\
\mbox{s.t.} \ & z \in \textrm{\chull}(P), \\
& Hz + Gu \leq h. 
\end{array}
\end{equation} }
$j \leftarrow 1$ and $Z^U_0 \leftarrow \infty$.\\
$V^1 \leftarrow V$.\\
Solve the
\emph{restricted DW master problem},
\begin{equation}\label{alg:gen:lb:dw}
\begin{array}{llll}
Z^L_j =& \max \ & c' V^j \lambda + d' u & \\
& \mbox{s.t.} \ & 1 \cdot \lambda = 1 & \\
&& H  V^j \lambda + G u \leq h & \\
&& \lambda \geq 0 & 
\end{array}
\end{equation}
Let ($\lambda^j,u^j)$ be an optimal solution to
problem \eqref{alg:gen:lb:dw}, and let $(z^j, u^j)$, be the corresponding feasible solution of \eqref{alg:prob:gen:dw}, where $z^j =  V^j \lambda^j$. Let $\pi^j$ be
an optimal dual vector corresponding to constraints $H  V^j \lambda + Gu \leq h$.\\
Solve the \emph{DW pricing problem}. For this, consider the problem,
\begin{equation}\label{alg:gen:ub:dw}
\begin{array}{llll}
L(\pi) = & \max \ & c' v  - \pi'(Hv - h) & \\
& \mbox{s.t.} \ & Av \leq b, & \\
&& v \in \{0,1\}^n &
\end{array}
\end{equation}
and let $v^j$ represent an optimal solution of $L(\pi^j)$. Let $Z^U_j \leftarrow \min \{ L(\pi^j), Z^U_{j-1} \}$.\\
\If {$Z^L_j = Z^U_j$}
{ $z^* \leftarrow z^j$. Stop.}
\Else
{
$V^{j+1} \leftarrow [V^j,v^j]$.\\
$j \leftarrow j+1$. Go to step 3. \\
}
\end{algorithm}
\end{figure}

Observe that the DW pricing problem (see \eqref{alg:gen:ub:dw}) is exactly the same as the GCG pricing problem. However, since optimizing over $\{ v: Av \leq b, \; v \in \{0,1\} \}$ is exactly the same as optimizing over \chull$(P)$, we will have at every iteration $j \geq 1$ that $L(\pi^j) \geq Z^{DW}$. Thus, the bounds will never cross as they might in GCG, and there will be no early termination condition.

The main difference between the algorithms is in the master problem, where both solve very different linear models. One would expect that solving the master problem for the DW decomposition method would be much faster. In fact, let $r_1$ and $r_2$ represent the number of rows in the $A$ and $H$ matrices, respectively. The DW master problem has $r_2 + 1$ rows, compared to the $r_1 + r_2$ rows in the GCG master problem. However, this is only the case because of the way in which problem \eqref{prob:gen-ip} is presented. If the first system of constraints, $Az \leq b$, instead has form $Az = b$, the algorithm can be modified to preserve this property. This equality form version of GCG, which we call GCG-EQ, is presented in Appendix \ref{sec:cg-eq}.

The discussion above suggests that for the GCG algorithm to outperform the DW algorithm, it would have to do less iterations. It is natural to expect that this would happen. After all, given a common set of columns, the feasible region of the GCG master problem is considerably larger than the corresponding feasible region of the DW master problem. That is, the dual solutions obtained by solving the GCG master problem should be ``better'' than those produced by the DW master problem, somehow driving the pricing problem to find the right columns quickly throughout the iterative process. This improved performance, of course, would only compensate on problems in which the DW algorithm has a tough time converging, and could come at the cost of a worse upper bound to the underlying integer programming problem. In fact, this slow convergence rate is exactly what happens in General Production Scheduling Problems (GPSPs), and is what motivates the BZ algorithm in the first place.

As column generation algorithms, the size of the master problem will be increasing by one in each iteration, both for the GCG and DW. At some point the number of columns could grow so large that solving the master problem could become unmanageable. An interesting feature of the DW algorithm is that this can be practically managed by removing columns that are not strictly necessary in some iterations. The key is the following Lemma, which is a direct result of well-known linear programming theory (Bertsimas and Tsisiklis \cite{bertsimas97}).

\begin{lemma}Let $(\lambda^*, u^*)$ represent an optimal basic solution of problem
\begin{equation*}
\begin{array}{llll}
Z(V) =& \max \ & c' V \lambda + d' u & \\
& \mbox{s.t.} \ & 1 \cdot \lambda = 1 & \\
&& H  V \lambda + G u \leq h & \\
&& \lambda \geq 0. & 
\end{array}
\end{equation*}
Then, $|\{ i : \lambda_i^* \neq 0 \}| \leq r_2+1$, where $r_2$ is the number of rows in matrix $H$.
\end{lemma}

In fact, what this Lemma says is that after $m$ iterations, no matter how large $m$, we can always choose to remove all but $r_2 + 1$ columns, thus making the problem smaller. As a means of ensuring that the resulting column generation algorithm does not cycle, a reasonable technique would be to remove columns only after the primal bound in the DW algorithm changes. That is, only in those iterations $k$ such that $Z^L_k > Z^L_{k-1}$. 

It should be noted that a variant of this Lemma is also true for the GCG-EQ algorithm (see Appendix \ref{sec:cg-eq}). However, there is no similar result for the GCG algorithm. That is, while unused columns can be discarded in GCG, there is no guarantee that the columns remaining in the master problem will be few in number.

\subsection{The BZ algorithm}\label{ssec:bz}

The BZ algorithm, originally proposed by Bienstock and Zuckerberg \cite{bienstock09, bienstock10} is a variant of the GCG algorithm that is specialized for General Precedence Constrained Problems (GPCPs). That is, the BZ algorithm assumes that constraints $Az \leq b$ correspond to precedence constraints having form $z_i - z_j \leq 0$, for $(i,j) \in I$, and bound constraints, having form $0 \leq z \leq 1$. This assumption allows the BZ to exploit certain characteristics of the optimal master solutions in order to speed up convergence. As we will see, the BZ algorithm is very effective when the number of rows in $H$ and the number of $u$ variables is small relative to the number of $z$ variables. It should be noted, however, that the optimal value of the problem solved by the BZ algorithm will have value $Z^{BZ} = Z^{LP}$. That is, the bound will be no tighter than that of the LP relaxation. This follows directly from the fact that $\{ z : z_i \leq z_j \quad \forall (i,j) \in I \}$ defines a totally unimodular system. 

The BZ algorithm is exactly as the GCG algorithm, with the exception of two changes:

First, the BZ algorithm uses a very specific class of generator matrices $V^{k}$. These matrices have two main properties. The first is that the linear space spanned by $V^{k+1}$ will always contain $z^k$, the optimal solution of the master problem in the previous iteration. The second is that the columns of $V^{k}$ are orthogonal $0$-$1$ vectors. That is, for every column $v^q$ of matrix $V^k$ define $I^q = \left\{ i \in \{1, \ldots, n\} : v^q_i \not= 0 \right\}$ (note that $I^q$ describes the \emph{support} of~$v^q$, for $q = 1,\ldots,k$). Then, $0$-$1$ vectors $v^q$ and $v^r$ are orthogonal if and only if their supports are disjoint, i.e., $I^q\cap I^r = \emptyset$.

An intuitive explanation for why this would be desirable is that, when $V^k$ has orthogonal $0$-$1$ columns, problem \eqref{prob:gen-lb} is equivalent to
\begin{equation}\label{prob:gen-lb-2}
\begin{array}{llll}
v^* = & \max \ & c'z + d'u & \\
& \mbox{s.t.} \ & Az \leq b & \\
&& Hz + Gu \leq h & \\
&& z_i = z_j & \forall i,j \in I^q, \; \forall q = 1,\ldots,k.
\end{array}
\end{equation}

That is, restricting the original feasible region to the linear space spanned by $V^k$ is equivalent to equating the variables corresponding to the non-zero entries of each column in $V^k$. In a combinatorial optimization problem, this resembles a contraction operation, which is desirable because it can lead to a significant reduction of rows and variables, while yet preserving the original problem structure.

The matrix $V^k$ used in the BZ algorithm is obtained by the following recursive procedure. Let $\hat v$ be the $0$-$1$ vector obtained from solving the pricing problem. Let $[V^k,\hat v]$ be the matrix obtained by appending column $\hat v$ to $V^k$. We would like to compute a matrix $V^{k+1}$ comprised of orthogonal $0$-$1$ columns such that, first, $Span(V^{k+1}) \supseteq Span([V^k,\hat v])$; and second, such that $V^{k+1}$ does not have too many columns. This can be achieved as follows. For two vectors $x,y \in \{0,1\}^n$, define $x \land y \in \{0,1\}^n$ such that $(x \land y)_i = 1$ if and only if $x_i = y_i = 1$. Likewise, define $x \setminus y \in \{0,1\}^n$ such that $(x \setminus y)_i = 1$ if and only if $x_i = 1$ and $y_i = 0$.  Assume that $V^k$ is comprised of columns $\{v^1,\ldots,v^r\}$. Let $V^{k+1}$ be the matrix made of the non-zero vectors from the collection: \[ \{ v^j \land \hat v : 1 \leq j \leq r \} \cup \{ v^j \setminus \hat v : 1 \leq j \leq r \} \cup \left\{ \hat v \setminus  \left( \sum_{j=1}^{k} v^j \right) \right\}. \] We call this procedure of obtaining the matrix $V^{k+1}$ from $V^k$ and $\hat v$ the \emph{refining} procedure. Note that in each iteration $k$, the refining procedure will produce a matrix $V^{k+1}$ with at most $2r+1$ columns. 

The second difference between BZ and GCG is that it introduces a mechanism for preventing an unmanageable growth in the number of columns used in the master problem.

Consider any nonzero vector $z \in \bbbr^n$. 
Let $\lambda_1,\dots,\lambda_d$ denote the distinct non-zero values of the components of~$z$, and 
for $i=1,\dots,d$ let $v^i \in \bbbr^n$ denote the indicator vector of the components of~$z$ equal to $\lambda_i$,
that is, $v^i$ has components $v^i_j = 1$ if $z_j = \lambda_i$, and $0$ otherwise.
Note that $1 \leq d \leq n$ and $v^1,\dots,v^d$ are orthogonal 0-1 vectors.
Thus we can write $z = \sum_{i=1}^{d} \lambda_i v^i$ and we say that $v^1,\dots,v^d$ define the \emph{elementary basis} of~$z$.

If at some iteration the number of columns in matrix $V^k$ is too large, the BZ algorithm will replace matrix $V^k$ with a matrix whose columns make up an elementary basis of the incumbent master solution $z^k$. 

Again, there are two possible problems with this. On the one-hand, it is possible that such a coarsification procedure leads to cycling. On the other-hand, it might still be the case that after applying coarsification procedure the number of columns remains prohibitively large.

To alleviate the first concern, the BZ algorithm applies the coarsification procedure in exactly those iterations in which the generation of a new column leads to a strict improvement of the objective function. That is, when $Z^L_k > Z^L_{k-1}$. The second concern is alleviated by the fact that under the BZ algorithm assumptions, the elementary basis associated to a basic feasible solution of problem \eqref{prob:gen-lb-2} will always have at most $r_2$ elements, where $r_2$ is the number of rows in matrix $H$ (see Bienstock and Zuckerberg \cite{bienstock09} for the original proof, or Appendix \ref{sec:frac-bz} for the proof adapted to our notation and extra variables).

A formal description of the column generation algorithm that results from the previous discussion is summarized in Algorithm \ref{alg:gen}.

\begin{figure}[h!]
\begin{algorithm}[H]
\caption{The BZ Algorithm.}\label{alg:gen}
\KwIn{A feasible linear programming problem of form
\begin{equation}\label{alg:prob:gen}
\begin{array}{llll}
Z^* = & \max \ & c'z + d'u & \\
& \mbox{s.t.} \ & Az \leq b & \\
&& Hz + Gu \leq h & 
\end{array}
\end{equation}
where constraints $Az \leq b$ correspond to precedence constraints having form $z_i - z_j \leq 0$, for $(i,j) \in I$, and bound constraints, having form $0 \leq z \leq 1$. A matrix $V$ whose columns are orthogonal $0$-$1$ vectors spanning at least one feasible solution of \eqref{alg:prob:gen}.\\}
\KwOut{
An optimal solution $(z^*,u^*)$ to
problem (\ref{alg:prob:gen}).\\}
$j \leftarrow 1$ and $Z^U_0 \leftarrow \infty$.\\
$V^1 \leftarrow V$.\\
Solve the
\emph{restricted BZ
master problem},
\begin{equation}\label{alg:gen:ub}
\begin{array}{llll}
Z^L_j =& \max \ & c' V^j \lambda + d' u & \\
& \mbox{s.t.} \ & A  V^j \lambda \leq b & \\
&& H  V^j \lambda + G u \leq h &
\end{array}
\end{equation}
Let ($\lambda^j,u^j)$ be an optimal solution to
problem \eqref{alg:gen:ub}, and let $(z^j, u^j)$, be the corresponding feasible solution of \eqref{alg:prob:gen}, where $z^j =  V^j \lambda^j$. Let $\pi^j$ be
an optimal dual vector
corresponding to constraints $H  V^j \lambda + Gu \leq h$.\\
Solve the \emph{BZ pricing problem}. For this, consider the problem,
\begin{equation}\label{alg:gen:lb}
\begin{array}{llll}
L(\pi) = & \max \ & c' v  -
\pi'(Hv - h) & \\
& \mbox{s.t.} \ & Av \leq b,  & 
\end{array}
\end{equation}
and let $v^j$ represent an optimal solution of $L(\pi^j)$. Let $Z^U_j \leftarrow \min \{ L(\pi^j), Z^U_{j-1} \}$.\\
\If {$Z^L_j = Z^U_j$}
{ $z^* \leftarrow z^j$. Stop.}
\Else
{
Obtain $V^{j+1}$, a matrix of orthogonal $0$-$1$ columns, by refining $V^j$ with $v^j$.\\
$j \leftarrow j+1$. Go to step 3. \\
}
\end{algorithm}
\end{figure}

\subsection{BZ algorithm speedups}\label{sec:speedups}

In this section we describe a number of computational techniques for improving the performance of the BZ algorithm. The computational techniques that we present for speeding up the master problem are generic, and can be used more generally in the GCG algorithm. However, the speedups that we present for the pricing problem are specific for solving generalized production scheduling problems (GPSPs), and thus, are specific for the BZ algorithm.

\subsubsection{Speeding up the BZ pricing algorithm}\label{ssec:pricing}

The pricing problem consists of solving a problem of the form
\begin{equation}\label{eq:pricing}
\begin{array}{llll}
& \max \ & \bar{c}'z & \\
& \mbox{s.t.} \ & z_i \leq z_j  & \forall (i,j) \in I \\
&& 0 \leq z \leq 1 &
\end{array}
\end{equation}
for some objective function $\bar c$, and a set of arcs $I$. This class of problems are known as \emph{maximum closure} problems. In order to solve \eqref{eq:pricing}, we use the Pseudoflow algorithm of \cite{hochbaum08}.
We use the following two features to speed-up this algorithm in the present context.

{\bf Pricing hot-starts [PHS].} The pricing problem is solved once per iteration of the BZ algorithm, each time with a different objective function. An important feature of the Pseudoflow algorithm is that it can be hot-started. More precisely, the Pseudoflow algorithm uses an internal data structure called a \emph{normalized tree} that can be used to re-solve an instance of a maximum closure problem after changing the objective function vector. Rather than solve each pricing problem from scratch, we use the normalized tree obtained from the previous iteration to hot-start the algorithm. Because the changes in the objective function are not componentwise monotonous from iteration to iteration, we need to re-normalize the tree each time. For more information on the detailed working of the Pseudoflow algorithm, with indications on how to incorporate hot-starts, see \cite{hochbaum08}.

{\bf Path contractions [PC].} Consider problem \eqref{eq:pricing}, and define an associated directed acyclic graph $G = (\mV, \mE)$ as follows. For each variable $z_i$ define a vertex $v_i$. For each precedence constraint $z_i \leq z_j$ with $(i,j) \in I$, define a directed arc $(v_i, v_j)$ in $\mE$ . We say that a directed path $P =
(v(1), v(2), \ldots, v(k))$ in $G$ is \emph{contractible} if it is a maximal path in $G$ such that (i) $k\geq 3$, and (ii) every internal vertex $v(2),\ldots,v(k-1)$ has both in-degree and out-degree equal to one.

Observe that the variables associated to this path can only take $k+1$ different combinations of $0$-$1$ values: either (i) $z_{v(i)} = 0$ for $i=1,\ldots,k$ and the total contribution of the nodes in~$P$ is zero; or else, (ii) there exists an index $j\in\{1,\dots,k\}$ such that $z_{v(i)} = 0$ for all $i = 1,\ldots,j-1$ and $z_{v(i)} = 1$ for all $i = j,\ldots, k$.

Since \eqref{eq:pricing} is a maximization problem, in an optimal integer solution to \eqref{eq:pricing} the contribution of the path will either be $0$ (this will happen when $z_{v(k)} = 0$), or the contribution will be $\max_{1\leq j\leq k} \sum_{i=j}^k \bar{c}_{v(i)}$ (which will happen when $z_{v(k)} = 1$).

This suggests the following arc-contracting procedure.

Let $j(P,\bar{c}) = \textrm{argmax}_{1\leq j\leq k} \sum_{i=j}^k \bar{c}_{v(i)}$ and define a new graph $\hat{G} = (\hat{\mV}, \hat{\mE})$ from $G$ by eliminating the vertices $v(2),\ldots,v(k-1)$ from~$\mV$ and the arcs incident to them, and then adding an arc that connects $v(1)$ and $v(k)$. Define $\hat c$ with $\hat c_v = \bar c_v$ for all vertices $v \notin \{v_1, \ldots, v_k \}$, $\hat c_{v(k)} = \sum_{i=j(P,\bar{c})}^k \bar{c}_{v(i)}$ and $\hat c_{v(1)} = \sum_{i<j(P,\bar{c})} \bar{c}_{v(i)}$.

Solving the pricing problem in this smaller graph $\hat{G}$ with the objective $\hat{c}$
gives an optimal solution to the original problem on graph~$G$ with objective~$\bar{c}$, with the same optimal objective value.

This procedure can be used to contract all contractible paths in $G$ in order to obtain, in some cases, a significantly smaller graph. In fact, problem instances with multiple destinations and batch constraints induce many contractible paths in the pricing problem graph. This is illustrated with an example in Figure \ref{fig:contraction}.

\begin{figure}[!ht]
    \centering
    \subfloat[Before path contraction]{{\includegraphics[width=50mm,scale=0.5]{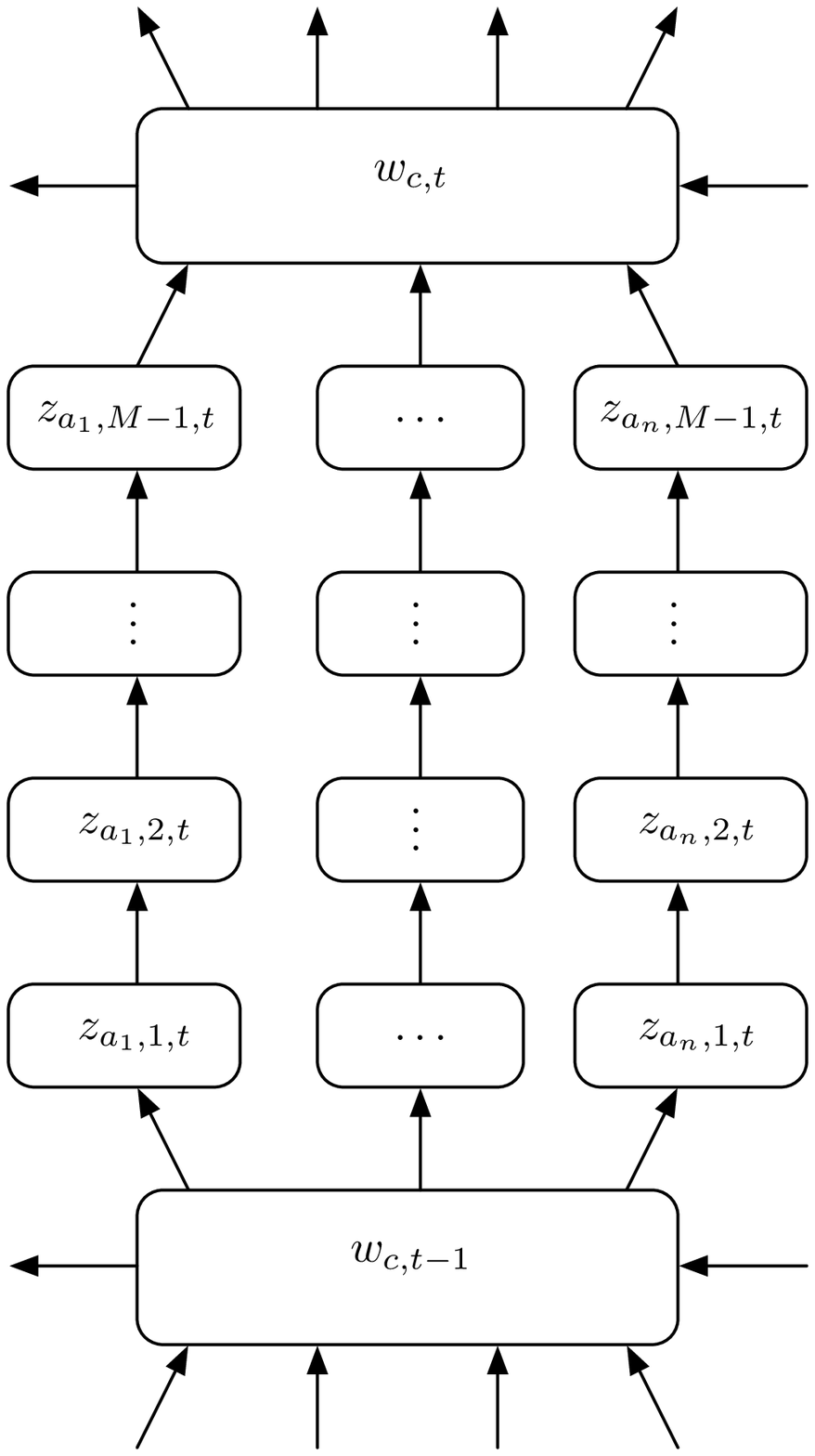}}}
    \qquad \qquad \qquad
    \subfloat[After path contraction]{{\includegraphics[width=50mm,scale=0.5]{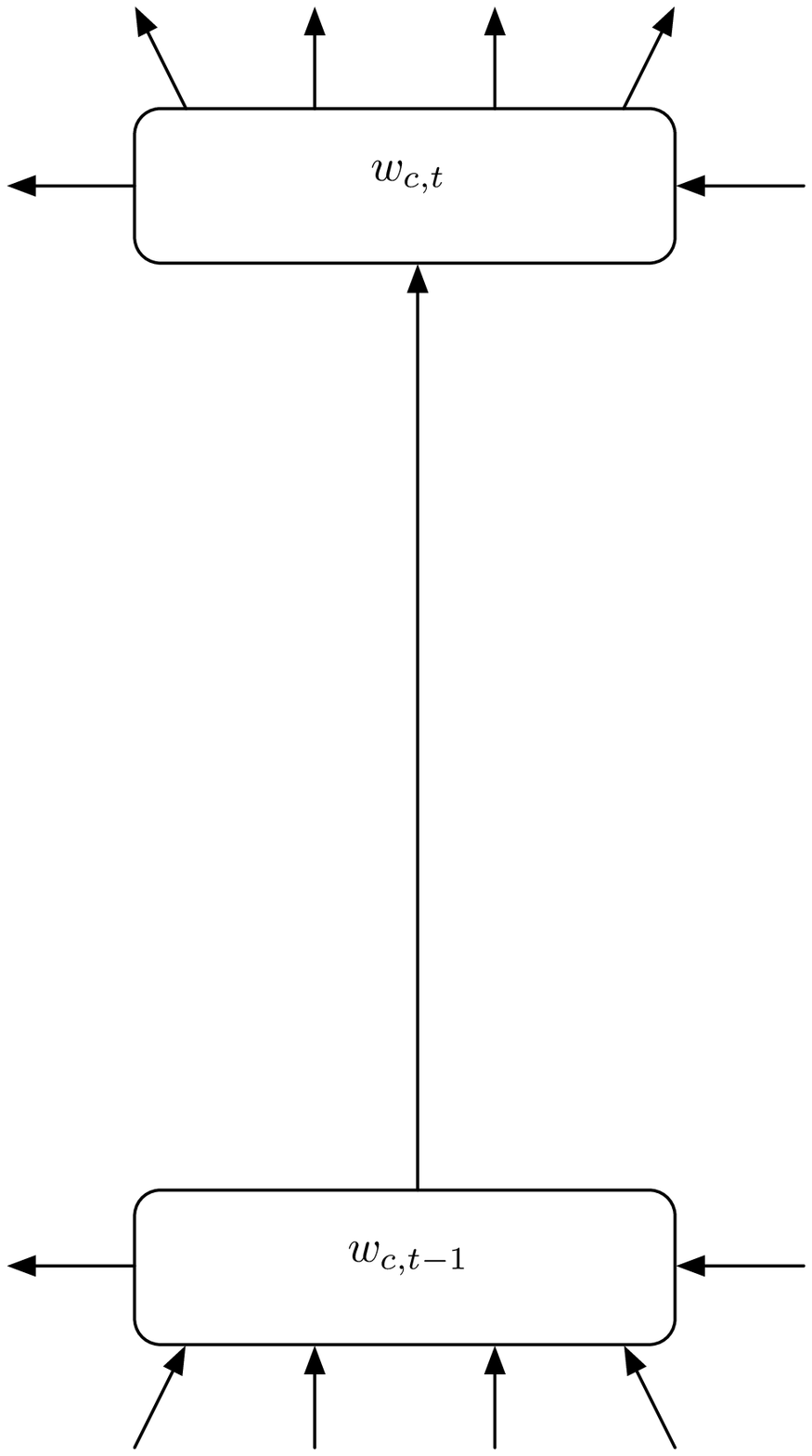}}}
    \caption{Example illustrating the possible effect of the path contraction speed-up. This example assumes a cluster $c$ comprised of activities $\{a_1,\ldots,a_n\}$. Each activity is assumed to have $M$ possible modes. The graph includes an arc between all pairs of variables for which there is a precedence relationship.}%
    \label{fig:contraction}%
\end{figure}

It should be noted that identifying and contracting paths only needs to be done in the first iteration. In all subsequent runs of the pricing problem, it is possible to use the same contracted graph, after updating the indices $j(P,\bar{c})$ and the objective function~$\hat c$.

\subsubsection{Speeding up the BZ master algorithm}\label{ssec:master}

The following ideas can be used to speed up solving the master problem \eqref{alg:gen:ub} in the BZ algorithm:

{\bf Starting columns [STCOL].} As input, the BZ algorithm requires an initial set of columns inducing a (possibly infeasible) solution of the problem. Such columns can be obtained by computing an elementary basis associated to a solution obtained with heuristics. 

If, when starting the algorithm, we do not have an initial set of columns describing a feasible  the problem, we can proceed as in the Phase-I simplex algorithm. For this, start with some arbitrary set of columns, and add ``artificial variables'' for each row. 

{\bf Master hot-starts [MHS].} Because of the refining procedure, the restricted BZ master problem \eqref{alg:gen:ub} in an iteration may be quite different from that in the previous iteration. To reduce its solve time, we can feed any simplex-based solver the solution from the previous iteration so that it can be used to attempt and identify a good feasible starting basis.

{\bf k-Step column management [k-Step].} In the original BZ algorithm, the coarsification procedure is applied in every iteration where there is a strict increase in the primal objective function value. Sometimes, however, it is better to keep the existing columns so as to improve the convergence rate. To avoid having too many columns in each iteration, we use what we refer to as a $k$-step column management rule. This rule, which requires as input an integer parameter $k$, works as follows: assume that we are at the $m$-th iteration of the BZ algorithm, where $m > k$. Further assume that in the previous iteration (iteration $m-1$) there was a strict increase in the primal objective function value. The column matrix $V^m$ used in the $m$-th iteration is built from scratch, by first obtaining an elementary basis associated to solution $z^{m-k}$, and then successively refining this basis with the vectors $v^{m-k+1}, v^{m-k+2}, \ldots, v^{m-1}$ (see Section \ref{ssec:bz}).

\section{Computational results}\label{sec:computations}

Our computational tests consider solving the linear relaxation of three different classes of problems. Specifically, we consider instances of OPPSP, OPPDP and RCPSP. The OPPSP and OPPDP instances are based on both real and hypothetical mine-planning instances obtained from industry partners, and from  the MineLib website \cite{espinoza13b}. In order to protect the confidentiality of the data sets obtained from our industry partners, we have renamed all of the instances using the names of Chilean volcanoes. In Table \ref{tab:instances} we present some characteristics of these instances. Note that for each instance we have a specification regarding the maximum number of time periods to consider. For each of the OPPSP problems we count with two versions: one with with less, and one with more clusters. We refer to these as the fine and course versions of each problem. These clusters correspond to what mining engineers refer to as a bench-phase, or increment (see \cite{hustrulid06} for formal definitions). We do not have cluster definition data for all of the problem instances, thus, the set of instances considered for OPPSP is a subset of the instances considered for OPPDP. The RCPSP instances that we consider are obtained from PSPlib \cite{kolisch97} (datasets \textit{j30}, \textit{j60}, \textit{j90} and \textit{j120}) and from \cite{debels2007decomposition} (dataset \textit{RG300}). It should be noted that these problems have a different objective function than the OPPDP and OPPSP problems. That is, these problems seek to minimize Makespan, rather than maximize net present value (NPV).

\begin{center}
\begin{table}[htbp]
\caption{Description of the different instances of the test set used in our computational study.\label{tab:instances}}
\begin{tabular}{|l|r|c|c|c||r|r|} \hline
Instance & Blocks (UPIT) &  Modes & Periods & Resources & \multicolumn{2}{c|}{Clusters}   \\ 
 &   &   &   &   &  Course  &  Fine \\ \hline
antuco             &  3\,525\,317  & 2 & 45 & 2 & --- & --- \\
calbuco            &    198\,248  & 3 & 21 & 1 & 324 & 548 \\
chaiten            &    287\,473  & 2 & 20 & 2 & 273 & 475 \\
dospuntas          &    897\,609  & 2 & 40 & 4 & --- & --- \\
guallatari         &     57\,958  & 3 & 21 & 3 & 272 & 460 \\
kd                 &     12\,154  & 2 & 12 & 1 &  53 &  93 \\
lomasblancas       &  1\,492\,024  & 3 & 45 & 3 & --- & --- \\
marvin             &      8\,516  & 2 & 20 & 2 &  56 &  98 \\
mclaughlin\_limit  &    110\,768  & 2 & 15 & 1 & 166 & 290 \\
palomo             &     97\,183  & 2 & 40 & 2 &  44 &  93 \\
ranokau            &    304\,977  & 2 & 81 & 2 & 186 & 296 \\
sm2                &     18\,388  & 2 & 30 & 2 & --- & --- \\
zuck\_large        &     96\,821  & 2 & 30 & 2 & --- & --- \\
zuck\_medium       &     27\,387  & 2 & 15 & 2 & --- & --- \\
zuck\_small        &      9\,399  & 2 & 20 & 2 & --- & --- \\
\hline
\end{tabular}
\end{table}
\end{center}

We evaluate the performance of the BZ algorithm, comparing it to the DW algorithm, and to a version of DW that incorporates the dual-stabilization techniques (DW-S) proposed by Pessoa et al. \cite{pessoa13}. For each of the algorithms we attempt to compute the linear relaxation solution of the test problems, stopping early if the relative difference between the master problem bound and the pricing problem bound is less than $10^{-6}$ (the default reduced cost tolerance used by the CPLEX simplex algorithm). We do not report the bounds obtained by each algorithm since they will all coincide with the value of the LP relaxation for each problem.

We also evaluate the performance of the speed-up techniques described in this paper. Specifically, we analyze the performance of Pricing hot-starts (PHS), Path contractions (PC), Starting columns (STCOL), Master hot-starts (MHS), and $k$-step column management ($k$-Step), with $k=10$. All of these speed-up techniques, with the exception of $k$-step (which is not applicable) are also implemented for the DW and DW-S algorithms. In fact, our code uses the exact same implementations of these speed-up techniques for the BZ, DW and DW-S algorithms. 

To assess the value of our proposed speed-up techniques we defined a \emph{default} set of speed-up-techniques to activate for each class of problems considered. For the OPPDP and OPPSP we defined the default speed-up-techniques to be PHS, MHS and PC. We found these to be the features that improved performance of the algorithm, without requiring a heuristic to generate a solution. For the RCPSP class of problems we use a different set of default speed-up options. Since we have to compute a feasible solution to each problem in order to define the number of time periods, we change the default settings so as to use this as a starting solution (STSOL). In addition, we observe that these problem instances have signficantly more time periods per activities when compared to the OPPSP and OPPDP instances. This resulted in a very different algorithm behavior that prompted us to include $k$-Step column management as a default option, as it improved problem performance. 

In order to assess the contribution of each individual speed-up technique in the algorithm, we turned each of these off (or on) to measure the impact from the change relative to that of the default settings. 

All algorithms were implemented using C programming language, using CPLEX\textsuperscript{\textregistered} 12.6 as optimization solver. The machines were running Linux 2.6.32 under x86\_64 architecture, with four eight-core Intel\textsuperscript{\textregistered} Xeon\textsuperscript{\textregistered} E5-2670 processors and with 128 Gb of RAM. For all results, we present normalized geometric means comparing the performance versus BZ algorithm under default configuration.

\subsection{Results for OPPDP}

Table \ref{tab:results_PD} shows the time required to solve each problem using the different proposed algorithms. The table also shows the number of iterations required by each algorithm, as well as the final number of columns generated by BZ. We do not report the bound generated by each relaxation, since they all generate the same bound for each instance.

\begin{center}
\begin{table}[htbp]
	\setlength{\tabcolsep}{5pt}
	\caption{Comparison between the different algorithms for OPPDP instances} \label{tab:results_PD}
\begin{tabular}{|l||r|r|r|r||r|r|r||r|} \hline
	\multirow{2}{*}{Instance} & \multicolumn{4}{c||}{Time (sec)} & \multicolumn{3}{c||}{Iterations}  &  \multirow{2}{*}{BZ cols} \\  \cline{2-8}
	& BZ & DW & DW+S & CPLEX & BZ & DW & DW+S  &  \\ \hline
	antuco            &208\,000 & 1\,542\,712 & 353\,313 &       - & 97 &2\,539 &   559 & 31\,640 \\
	calbuco          &  1\,480 &      6\,715 &   3\,602 &       - & 34 &   243 &   132 &  1\,876 \\ 
	chaiten            &  2\,620 &     35\,339 &   9\,625 &       - & 34 &   884 &   259 &  5\,800 \\ 
	dospuntas            & 36\,100 &    451\,515 & 108\,105 &       - & 52 &1\,542 &   385 &112\,943\\
	guallatari          &     143 &      1\,668 &      494 &       - & 30 &   444 &   158 &  2\,421 \\ 
	kd                &       7 &          22 &       16 & 15\,560 & 23 &   115 &    90 &     365 \\ 
	lomasblancas      & 86\,600 &    521\,279 & 208\,256 &       - & 75 &   773 &   360 &  5\,635 \\
	marvin            &       9 &          28 &       17 & 34\,601 & 27 &   154 &    99 &     206 \\ 
	mclaughlin\_limit &     167 &      1\,096 &      539 &       - & 27 &   208 &   102 &  1\,273 \\ 
	palomo          &  1\,290 &      9\,550 &   3\,637 &       - & 45 &   587 &   264 &  2\,089 \\ 
	ranokau           & 192\,000& 1\,526\,274 & 210\,506 &       - &108 &7\,261 &1\,025 &124\,406  \\
	sm2               &      44 &         256 &      109 &     254 & 52 &   874 &   281 &  3\,594 \\ 
	zuck\_large       &     634 &      3\,564 &   1\,761 &       - & 46 &   416 &   194 &  3\,195 \\ 
	zuck\_medium      &      40 &         112 &       93 &954\,405 & 28 &   107 &    82 &     229 \\ 
	zuck\_small       &      11 &          37 &       24 & 56\,165 & 29 &   150 &   113 &     294 \\ \hline
	Norm. G. Mean     &       1 &        5.98 &     2.42 &       - &  1 &   11.8&   4.9 &   \\ \hline
\end{tabular}
\end{table}
\end{center}

As can be seen, DW is nearly six times slower than BZ. Using stabilization techniques, the DW-S is able to signficantly reduced the time required to reach the optimality condition. In fact, stabilization techniques reduce the time of DW in a 60\%, but it is still $2.42$ times slower than BZ. This can be explained by the number of iterations required to converge by each algorithm. Even though a BZ iteration is, in average, $2.1$ times slower than a DW iteration, DW and DW+S require 11.8 and 4.9 times the number of iterations of BZ to converge. This is possible because the BZ algorithm can produce a large number of useful columns in few iterations. In fact, the number of columns generated by BZ is considerably larger than the number of columns produced by DW and DW-S (the number of columns for these algorithms is equal to the number of iterations). Finally, note that CPLEX is only able to solve the five smallest instances. 

In Table \ref{tab:results_PD_gaps} we show what happens when we change the tolerance used to define optimality in all of the algorithms. It can be seen that BZ algorithm still outperforms the DW and DW-S algorithms after reducing the target optimality gap to $10^{-4}$ and $10^{-2}$. However, the performance difference narrows as the target gap becomes smaller. The resulting times, normalized to the time of BZ under default setting, are presented in Table \ref{tab:results_PD_gaps}.

\begin{center}
\begin{table}[htbp]
	\caption{Normalized time required for different optimality gaps for OPPDP instances} \label{tab:results_PD_gaps}
\begin{tabular}{|l||r|r|r||r|r|r||r|r|r|} \hline
	\multirow{2}{*}{Instance} & \multicolumn{3}{c||}{$10^{-6}$} & \multicolumn{3}{c||}{$10^{-4}$}  &  \multicolumn{3}{c|}{$10^{-2}$} \\  \cline{2-10}
	& BZ & DW & DW+S  & BZ & DW & DW+S  & BZ & DW & DW+S \\ \hline
	antuco            & 1 & 7.42 & 1.70 & 0.96 & 6.28 & 1.49 & 0.93 & 4.63 & 1.20 \\
	calbuco          & 1 & 4.54 & 2.43 & 0.92 & 3.49 & 1.94 & 0.74 & 1.55 & 1.08 \\ 
	chaiten            & 1 &13.49 & 3.67 & 0.95 &10.53 & 2.92 & 0.79 & 4.30 & 1.57 \\ 
	dospuntas            & 1 &12.51 & 2.99 & 0.83 & 7.77 & 2.16 & 0.67 & 2.40 & 1.15 \\
	guallatari          & 1 &11.66 & 3.45 & 0.86 & 7.35 & 2.44 & 0.70 & 2.60 & 1.37 \\ 
	kd                & 1 & 3.05 & 2.25 & 0.90 & 2.40 & 1.78 & 0.76 & 1.13 & 1.18 \\ 
	lomasblancas      & 1 & 6.02 & 2.40 & 0.95 & 5.04 & 2.05 & 0.87 & 3.26 & 1.38 \\
	marvin            & 1 & 3.19 & 1.99 & 0.91 & 2.26 & 1.47 & 0.81 & 0.94 & 0.81 \\ 
	mclaughlin\_limit & 1 & 6.56 & 3.23 & 0.86 & 4.75 & 2.47 & 0.68 & 1.89 & 1.47 \\ 
	palomo          & 1 & 7.40 & 2.82 & 0.98 & 6.10 & 2.33 & 0.88 & 2.69 & 1.50 \\ 
	ranokau           & 1 & 7.95 & 1.10 & 0.82 & 4.79 & 0.86 & 0.63 & 1.87 & 0.50 \\
	sm2               & 1 & 5.88 & 2.50 & 0.95 & 3.77 & 1.89 & 0.88 & 2.03 & 1.35 \\ 
	zuck\_large       & 1 & 5.62 & 2.78 & 0.96 & 4.42 & 2.28 & 0.84 & 2.21 & 1.45 \\ 
	zuck\_medium      & 1 & 2.80 & 2.31 & 0.87 & 2.15 & 1.78 & 0.66 & 1.23 & 0.97 \\ 
	zuck\_small       & 1 & 3.22 & 2.13 & 0.86 & 2.20 & 1.52 & 0.77 & 1.24 & 0.85 \\ \hline
	Norm. G. Mean     & 1 & 5.98 & 2.42 & 0.90 & 4.36 & 1.89 & 0.77 & 2.04 & 1.14 \\ \hline
\end{tabular}
\end{table}
\end{center}

Finally, we study the impact of the different speed-up techniques on BZ. From the default settings, we disable the speed-ups PHS, MHS and PC, one-by-one, and also, we disable all three of them together. We also run BZ under our default setting after adding a starting column (STCOL) generated using the Critical Multiplier algorithm (See \cite{chicoisne12}), and after enabling the k-Step feature. The resulting times, normalized to the time of BZ under default setting, are presented in Table \ref{tab:results_DP_features}.

\begin{center}
\begin{table}[htbp]
	\caption{Normalized time required for BZ algorithm with/without features for OPPDP instances} \label{tab:results_DP_features}
	\begin{tabular}{|l||c|c|c|c|c|c|c|} \hline
	Instance & Default&No & No & No  &No PHS/PC& Default & Default   \\
           &        &PHS& PC & MHS &MHS      &+ STCOL  & + k-Step\\ \hline
	antuco            & 1 & 1.57 & 1.60 & 0.98 & 2.50 & 0.57 & 2.56 \\
	calbuco           & 1 & 1.48 & 1.77 & 1.10 & 2.78 & 0.85 & 1.09 \\
	chaiten           & 1 & 1.38 & 1.49 & 1.01 & 1.83 & 0.62 & 1.20 \\
	dospuntas         & 1 & 1.36 & 1.57 & 1.03 & 1.75 & 0.73 & 1.52 \\
	guallatari        & 1 & 1.47 & 2.18 & 1.21 & 2.88 & 1.00 & 0.99 \\
	kd                & 1 & 1.30 & 1.63 & 1.00 & 2.22 & 0.69 & 1.51 \\
	lomasblancas      & 1 & 1.84 & 1.68 & 1.01 & 2.94 & 0.25 & 1.98 \\
	marvin            & 1 & 1.32 & 1.27 & 0.97 & 2.47 & 0.91 & 1.49 \\
	mclaughlin\_limit & 1 & 1.27 & 1.78 & 0.94 & 2.06 & 0.77 & 1.20 \\
	palomo            & 1 & 1.60 & 1.47 & 0.98 & 2.21 & 0.94 & 1.53 \\
	ranokau           & 1 & 0.93 & 1.04 & 0.97 & 1.15 & 0.22 & 0.57 \\
	sm2               & 1 & 1.33 & 1.93 & 1.16 & 2.35 & 1.11 & 2.06 \\ 
	zuck\_large       & 1 & 1.15 & 1.60 & 1.00 & 2.08 & 0.62 & 1.31 \\
	zuck\_medium      & 1 & 1.75 & 1.59 & 1.20 & 2.74 & 0.84 & 1.13 \\
	zuck\_small       & 1 & 1.32 & 1.72 & 0.95 & 2.41 & 0.92 & 1.27 \\ \hline
	Norm. G. Mean.    & 1 & 1.39 & 1.60 & 1.03 & 2.24 & 0.68 & 1.35 \\ \hline
\end{tabular}
\end{table}
\end{center}

We can see that all speeding features provide, in average, an improvement on the time required by BZ to converge. However, the individual performance of each feature differs instance by instance. The most important speed-up technique is path contraction (PC), as disabling this feature increases the time required to converge by 60\%. Since the reduction in number of variables of this speed-up in OPPDP is proportional to the number of modes of the problem, this feature is particularly important for the Guallatari instance, which has 3 different modes. Disabling these three features, the total time required to converge more than doubles. On the other hand, if we start with a preliminary set of starting columns, the time required is decreased by 32\%.  Finally, we see that k-Step does not improve the convergence time, making the algorithm run 35\% slower. In fact, it makes every problem converge slower, with the exception of the ranokau instance.

\subsection{Results for OPPSP}

For these instances we use the same default settings as those used for OPPDP. We note that these problems are considerably smaller than the OPPSP problems. This is both in the number of variables and precedence constraints. All algorithms are able to solve these problems in a few hours, with the exception of the \emph{ranokau} instance, which CPLEX fails to due to the memory limit (128Gb). We present the results in Table \ref{tab:results_PS}.

\begin{center}
\begin{table}[htbp]
	\caption{Comparison between the different algorithms for OPPSP instances} \label{tab:results_PS}
\begin{tabular}{|l||r|r|r|r||r|r|r||r|} \hline
	\multirow{2}{*}{Instance} & \multicolumn{4}{c||}{Time (sec)} & \multicolumn{3}{c||}{Iterations}  &  \multirow{2}{*}{BZ cols} \\  \cline{2-8}
	& BZ & DW & DW+S & CPLEX & BZ & DW & DW+S  &  \\ \hline
	calbuco           &    88 &      90 &    122 & 56\,102 & 32 &    96 &  95 &   167 \\ 
	chaiten           &    86 &     174 &    184 &  4\,595 & 34 &   189 & 149 &   392 \\ 
	guallatari        &    29 &      39 &     50 & 11\,757 & 30 &   117 & 102 &   188 \\ 
	kd                &     1 &       1 &      1 &       5 & 18 &    44 &  40 &    70 \\ 
	marvin            &     2 &       2 &      2 &      17 & 27 &    70 &  73 &    82 \\ 
	mclaughlin\_limit &    14 &      19 &     22 &  2\,310 & 25 &    78 &  73 &   100 \\ 
	palomo            &    49 &      59 &     95 &     429 & 44 &   138 & 137 &   160 \\ 
	ranokau           &2\,034 &  5\,102 &   4084 &       - &186 &1\,461 & 781 &2\,108\\ \hline
	Norm. G. Mean     &     1 &    1.38 &   1.63 &   54.57 &  1 &  3.64 &3.17 &   \\ \hline
\end{tabular}
\end{table}
\end{center}

We can see that in this class of problems the performance of DW and DW-S is more similar to that of BZ. This is probably explained by the fact that the number of iterations required by DW and DW-S is much smaller. Note also that stabilization techniques for DW only marginally reduce the number of iterations required by DW to converge, resulting in that DW-S is 18\% slower than DW.

\begin{center}
\begin{table}[htbp]
	\caption{Normalized time required for BZ algorithm with/without features for OPPSP instances} \label{tab:results_PS_features}
	\begin{tabular}{|l||c|c|c|c|c|c|} \hline
	Instance & Default&No & No & No  & Default & Default   \\
           &        &PHS& PC & MHS &+ STCOL  & + k-Step\\ \hline
	calbuco           & 1 & 0.92 & 9.96 & 1.03 & 0.82 & 2.10 \\
	chaiten           & 1 & 0.95 &14.63 & 1.04 & 0.81 & 2.06 \\
	guallatari        & 1 & 1.01 & 4.91 & 1.05 & 0.73 & 1.77 \\
	kd                & 1 & 1.04 & 5.63 & 1.05 & 0.65 & 1.59 \\
	marvin            & 1 & 0.90 & 3.51 & 0.92 & 0.68 & 1.26 \\
	mclaughlin\_limit & 1 & 0.99 &14.65 & 0.93 & 0.98 & 1.91 \\
	palomo            & 1 & 0.87 &13.33 & 0.97 & 0.77 & 1.83 \\
	ranokau           & 1 & 1.03 &19.48 & 1.03 & 0.78 & 2.37 \\ \hline
	Norm. G. Mean     & 1 & 0.96 & 9.25 & 1.00 & 0.77 & 1.83 \\ \hline
\end{tabular}
\end{table}
\end{center}

Comparing the impact of the different features on BZ (see Table \ref{tab:results_PS_features}), we see that the most important feature is, again, path contraction (PC). Disabling this feauture makes the algorithm run almost $10$ times slower. Similarly to the OPPDP problems, providing starting columns to BZ reduces the time by $23\%$, and introducing $k$-Step column management makes the problem run slower ($83\%$). 

\subsection{Results for RCPSP instances}

In order to formulate the RCPSP instances we need a maximum number of time periods to consider. Since the objective of these problems is to minimize Makespan, the number of time periods should be an upper bound on the number of periods required to schedule all of the activities. Such a bound can be computed by running any heuristic to compute any feasible integer solution. For this purpose, we use a greedy TOPOSORT heuristic~\cite{chicoisne12}, which takes fractions of a second to solve for all of our instances. 

Table \ref{tab:results_RCPSP} describes the performance of the different algorithms on our RCPSP test instances. Note that we only consider instances that are solved by CPLEX in more than 1 second, obtaining a total of 1575 instances. The running times presented in the table are geometric means over instances in the same dataset.

\begin{center}
\begin{table}[htbp]
	\caption{Comparison between the different algorithms for RCPSP instances} \label{tab:results_RCPSP}
\begin{tabular}{|lr||r|r|r|r||r|r|r|} \hline
&	& \multicolumn{4}{c||}{Time (sec)} & \multicolumn{3}{c|}{Iterations} \\ \cline{3-9}
\multicolumn{2}{|l||}{Dataset \textdagger} & BZ & DW & DW+S & CPLEX & BZ & DW & DW+S  \\ \hline
j30  & (51 inst.) & 1.23 & 4.71 & 1.91 &  1.61 & 141.4 & 650.4 & 312.9 \\%& 105.1\\
j60  &(152 inst.) & 0.98 & 8.19 & 2.55 &  3.77 &  68.5 & 477.0 & 232.1 \\%& 124.2\\
j90  &(293 inst.) & 0.54 & 2.27 & 1.07 &  5.33 &  23.3 & 115.4 &  73.4 \\%&  69.7\\ 
j120 &(599 inst.) & 3.95 &33.13 &11.94 & 31.78 &  58.2 & 440.5 & 230.5 \\%& 179.1\\ 
RG300&(480 inst.) &22.86 &43.76  &24.87 & 240.51\textsuperscript{\textasteriskcentered}&  91.1 & 393.9 & 260.5 \\ \hline%&  99.6\\ \hline
\multicolumn{2}{|l||}{Norm. G. Mean}     & 1    & 4.76 & 2.00 &  7.25 &   1   &   5.8 &   3.3 \\ \hline%& \\ \hline
\end{tabular}\\
\begin{footnotesize}
	\textdagger: We only consider instances which took CPLEX more than a second to solve.\\
	\textasteriskcentered: We only consider the 438 instances which were solved within 48 hours.\\
\end{footnotesize}
\end{table}
\end{center}

Table \ref{tab:results_RCPSP} shows that BZ is again faster than the other algorithms when solving RCPSP instances. In fact, it is 2 times faster than DW+S and 4.7 faster than DW.  Note that this difference is particularly large for the \emph{j120} instances from PSPLIB repository, where DW and DW+S run 8.4 and 3 times slower, respectively. It would seem that the performance of these algorithms is greatly dependent on problem structure. This can be seen when considering the \emph{RG300} instances, which are generated in a different way. In these instances, DW with stabilization is only marginally slower than BZ. 

Table \ref{tab:results_RCPSP_features} shows how much the performance of the BZ algorithm is affected by turning off each of the default speed up features. It is interesting to note that on RCPSP instances the $k$-Step column management rule is very important. Turning it off makes the problem run, in average, $2.51$ times slower. This is in stark contrast to what happens with the OPPSP and OPPDP instances, where activating the $k$-Step feature actually makes the problem run slower. We speculate that this is due to the fact that the RCPSP problems have a significantly greater number of time periods per activity than the OPPSP and OPPDP instances. This results in a problem with significantly more resource consumption constraints per variable. It is also interesting to note that disabling the PC and MHS features actually makes BZ run faster in the RCPSP instances. We speculate that the MHS feature does not improve performance because, having enabled the $k$-Step feature as well, succesive master problem formulations significantly differ from each other. We speculate that the PC feature does not improve performance because in RCPSP instances there are not as many paths to contract due to the structure of the precedence graphs. This is due to the fact that there are no clusters and that there is just a single mode per activity.

\begin{center}
\begin{table}[htbp]
	\caption{Normalized time required for BZ algorithm with/without features for RCPSP instances} \label{tab:results_RCPSP_features}
	\begin{tabular}{|lr||c|c|c|c|c|c|} \hline
	     &       & Default & No   & No   & No   & No   & No     \\
 \multicolumn{2}{|l||}{Dataset \textdagger}  & & PHS  & PC   & MHS  &k-Step& STCOL \\ \hline
	j30  & (51 inst.) & 1      & 1.30 & 0.78 & 0.80 & 0.74 & 1.44 \\ 
	j60  &(152 inst.) & 1      & 1.53 & 0.99 & 1.03 & 1.44 & 1.91 \\ 
	j90  &(293 inst.) & 1      & 1.15 & 0.86 & 0.91 & 1.51 & 1.80 \\ 
	j120 &(599 inst.) & 1      & 1.40 & 0.92 & 0.95 & 2.39 & 1.23 \\ 
	RG300&(453 inst.) & 1      & 1.22 & 0.87 & 0.89 & 5.78 & 1.04 \\ \hline
	\multicolumn{2}{|l||}{Norm. G. Mean} & 1      & 1.30 & 0.90 & 0.93 & 2.51 & 1.33 \\ \hline
	\end{tabular}\\
	\begin{footnotesize}
	\textdagger: We only consider instances for which the BZ algorithm finished in 48 hours in all settings.\\
\end{footnotesize}
\
\end{table}
\end{center}

\subsection{Concluding remarks}

We summarize with three important conclusions. First, the BZ algorithm significantly outperforms DW and DW+S on all GPCP problem classes. Second, the speed-up features proposed in this article significantly improve the performance of the BZ algorithm. Third, the algorithm and speed-up performances greatly depend on the problem classes that are being considered. These conclusions suggest that the BZ algorithm with the proposed speed-ups is a technique that can be used to effectively be used to compute the LP relaxation solution of different classes of scheduling problems. Such an algorithm might be useful as a means to provide upper bounds for scheduling problems, or as a part of the many rounding heuristics that have recently been proposed, or eventually, as part of a branch-and-bound solver. In addition, they suggest that the BZ algorithm's potential use for other classes of problems should also be further studied.

\bibliographystyle{spmpsci}      % basic style, author-year citations
\bibliography{orbib}

\newpage
\appendix
\noindent {\Large \bf Appendix}

\section{Resource constraints in GPSP reformulation}\label{sec:resources}

In this section we provide a detailed proof on how to derive \eqref{psp2:resources} from \eqref{eq:psp:resources}. This is needed in order to reformulate an instance of a General Production Scheduling Problem (GPSP), described in Section \ref{sec:formulation}, as an instance of a General Precedence Constrained Problem (GPCP), presented in Section \ref{sec:reformulation}. For the benefit of the reader, we recall \eqref{psp2:resources} and \eqref{eq:psp:resources}.

\begin{proposition}
Consider inequality \eqref{eq:psp:resources}
\begin{equation*}
\sum\limits_{a \in \mA} \sum\limits_{m \in \mM_a}  q_{r,a,m} \sum\limits_{s=\max\{1,t-d_{a,m}+1\}}^{t} y_{a,m,s} \leq Q_{r,t},
\end{equation*}
as defined in Section \ref{sec:formulation}. If we apply the variable substitution scheme introduced in Section \ref{sec:reformulation}:
\begin{align*}
y_{a,m,t} & = z_{a,m,t} - z_{a,m-1,t} & \qquad \forall a \in \mA, \; m = 2, \ldots, M_a, \; t \in \mT, \\
y_{a,1,t} & = z_{a,1,t} - z_{a,M_a,t-1} & \qquad \forall a \in \mA, \; t = 2, \ldots, T, \\
y_{a,1,1} & = z_{a,1,1} & \qquad \forall a \in \mA,
\end{align*}
we obtain inequality \eqref{psp2:resources}:  
\begin{equation*}
\sum\limits_{a \in \mA} \sum\limits_{m \in \mM_a} \sum\limits_{s\in \mT } \tilde{q}^{r,t}_{a,m,s} z_{a,m,s} \leq Q_{r,t},
\end{equation*}
\noindent where \[ \tilde{q}^{r,t}_{a,m,s} = \left\{ 
\begin{array}{ll} 
q_{r,a,m} \mathbf{1}_{[t-d_{a,m}+1,t]}(s)  - q_{r,a,m+1}  \mathbf{1}_{[t-d_{a,m+1}+1,t]}(s) & \text{ if }  m < M_a \\
q_{r,a,M_a} \mathbf{1}_{[t-d_{a,M_a}+1,t]}(s) - q_{r,a,1} \mathbf{1}_{[t-d_{a,1},t-1]}(s) & \text{ if } m = M_a \end{array} \right.\]
\end{proposition}
\begin{proof}
\begin{eqnarray*}\label{eq:psp:resources_bbb}
Q_{r,t} &\geq & \sum\limits_{a \in \mA} \sum\limits_{m \in \mM_a}  q_{r,a,m} \sum\limits_{s=\max\{1,t - d_{a,1} + 1\}}^{t} y_{a,m,s}\\
&=& \sum\limits_{a \in \mA} \left(  \sum\limits_{m=2}^{M_a}   \sum\limits_{\ s=\max\{1,t - d_{a,m} + 1\}}^{t} q_{r,a,m} y_{a,m,s} +   \sum\limits_{s=\max\{1,t - d_{a,1} + 1\}}^{t} q_{r,a,1} y_{a,1,s} \right)\\
&=& \sum\limits_{a \in \mA} \left(  \sum\limits_{m=2}^{M_a}   \sum\limits_{\ s=\max\{1,t - d_{a,m} + 1\}}^{t} q_{r,a,m} (z_{a,m,s}-z_{a,m-1,s}) \right. \\
&& \left. +   \sum\limits_{s=\max\{1,t - d_{a,1} + 1\}}^{t} q_{r,a,1} z_{a,1,s} - \sum\limits_{s=\max\{2,t - d_{a,1} + 1\}}^{t} q_{r,a,1} z_{a,M_a,s-1} \right)\\
&=& \sum\limits_{a \in \mA} \left(  \sum\limits_{m=2}^{M_a}   \sum\limits_{\ s=\max\{1,t - d_{a,m} + 1\}}^{t} q_{r,a,m} z_{a,m,s}  - \sum\limits_{m=1}^{M_a-1}   \sum\limits_{\ s=\max\{1,t - d_{a,m+1} + 1\}}^{t} q_{r,a,m+1} z_{a,m,s} \right. \\
&& \left. + \sum\limits_{s=\max\{1,t - d_{a,1} + 1\}}^{t} q_{r,a,1} z_{a,1,s} - \sum\limits_{s=\max\{1,t - d_{a,1}\}}^{t-1} q_{r,a,1} z_{a,M_a,s} \right)\\
&=& \sum\limits_{a \in \mA} \left(  \sum\limits_{m=1}^{M_a}   \sum\limits_{\ s=\max\{1,t - d_{a,m} + 1\}}^{t} q_{r,a,m} z_{a,m,s} - \sum\limits_{m=1}^{M_a-1}   \sum\limits_{\ s=\max\{1,t - d_{a,m+1} + 1\}}^{t} q_{r,a,m+1} z_{a,m,s}  \right.\\
&& \left.  -  \sum\limits_{s=\max\{1,t - d_{a,1}\}}^{t-1} q_{r,a,1} z_{a,M_a,s} \right). \\
\end{eqnarray*}
\end{proof}

\section{Equality form version of the GCG algorithm}\label{sec:cg-eq}

In this section we present the GCG algorithm when, instad of the form \eqref{prob:gen-ip}, the mixed integer problem to be solved has form,
\begin{equation}\label{prob:gen-ip-eq}
\begin{array}{rl}
Z^{IP} = \max \ \; & c'z + d'u \\
\mbox{s.t.} \ \; & Az = b, \\
& Hz + Gu \leq h,  \\
& z \in \{0,1\}^n.
\end{array}
\end{equation}
That is, when the first set of inequalities have been replaced with equalities. As before, define
\[ P = \{ z : Az = b, \; z \in \{0,1\} \}. \]

\noindent We present a column generation algorithm for computing the optimal solution of
\begin{equation}\label{prob:lin-eq}
\begin{array}{rl}
Z^{LIN} = \max \ \; & c'z + d'u \\
\mbox{s.t.} \ \;& Az = b, \\
& Hz + Gu \leq h, \\
& z \in \textrm{\lhull}(P).
\end{array}
\end{equation}

Let $v^o$ be such that $Av^o = b$. For each $k \geq 1$ let $V^k$ be a matrix comprised of linearly independent points $\{v^1,\ldots,v^k\}$ such that $Av^i = 0$ for $i=1,\ldots,k$. Define,

\begin{equation}\label{prob:gen-lb-eq}
\begin{array}{llll}
Z(V) = & \max \ & c'v^o + c'V \lambda + d'u & \\
& \mbox{s.t.} \ & H V \lambda + G u \leq h - H v^o & (\pi)
\end{array}
\end{equation}

Begin each iteration by solving problem \eqref{prob:gen-lb-eq} with $V = V^k$. Define $Z^L_k = Z(V^k)$ and let $(\lambda^k, u^k)$ and $\pi^k$ be the corresponding optimal primal and dual solutions respectively. Define $z^k = v^o  + V^k \lambda^k$ and let,

\begin{equation}\label{prob:gen-ub-eq}
\begin{array}{llll}
L(\pi) = & \max \ & c' v - \pi(Hv - h) & \\
& \mbox{s.t.} \ & Av = b, & \\
&& v \in \{0,1\}^n. &
\end{array}
\end{equation}

Observe that since $\pi^k G = d'$, we have: 

\begin{equation}\label{prob:gen-ub-eq-2}
\begin{array}{llll}
L(\pi^k) = & \max \ & c' v + d'u - \pi^k(Hv + Gu - h) & \\
& \mbox{s.t.} \ & Av = b, & \\
&& v \in \{0,1\}^n. &
\end{array}
\end{equation}

Thus, $L(\pi^k)$ is a relaxation of \eqref{prob:gen-ip-eq} obtained by penalizing constraints $Hz + Gu \leq h$ with $\pi^k \geq 0$. This implies $L(\pi^k) \geq Z^{IP}$ for all iterations $k$.

Solve problem \eqref{prob:gen-ub-eq} with $\pi = \pi^k$ and let $\hat v$  be the corresponding optimal solution. Define $Z^U_k = \min \{L(\pi^i) : i = 1,\ldots, k\}$. As before, we will have that $Z^U_k \geq Z^{IP}$ and $Z^L_k \leq Z^{LIN}$. 

If, at any iteration, we have that $Z^L_k \geq Z^U_k$ we can halt the algorithm. By defining $Z^{UB} = Z^U_k$ we will have that $Z^{IP} \leq Z^{UB} \leq Z^{LIN}$. 

On the other hand whenever $Z^L_k < Z^U_k$ we will have that column $\hat v - v^o$ has positive reduced cost in the master problem. Thus, letting $v^{k+1} = \hat v - v^o$ and $V^{k+1} = [V^k, v^{k+1}]$ we can continue iterating the algorithm.

\section{Distinct fractional values Lemma}\label{sec:frac-bz}

In this section we show that the BZ master problem will have bounded number of distinct fractional values. This fact is important in Section \ref{ssec:bz}, in order to argue why the use of the \emph{elementary basis} can reduce the number of columns in the BZ algorithm.

Consider the setting for the BZ algorithm, i.e, a problem of the form
\begin{equation}\label{appendix:BZ}
\begin{array}{llll}
Z^{BZ} = & \max \ & c'z + d'u & \\
& \mbox{s.t.} \ & Az \leq b & \\
&& Hz + Gu \leq h & 
\end{array}
\end{equation}
where $\{ z :  Az \leq b\} = \{ z :  z_i \leq z_j\ \forall \ (i,j) \in I,\ 0\leq z \leq 1\}$. It is well known that $A$ defines a totally unimodular matrix in such case, thus all extreme points of $\{z : Az\leq b\}$ are 0-1 vectors.\\

Recall that $z$ is an $n$-dimensional vector, and let $m$ be the dimension of $u$. In this section we prove the following result:

\begin{lemma}\label{lemma:fractional}
Let $(\bar z, \bar u)$ be an extreme point of \eqref{appendix:BZ}, and let $q$ be the number of rows of $H$ (and $G$). Then the number of distinct fractional values of $\bar z$ is at most $q-m$.
\end{lemma}

Note that by assuming an extreme point $(\bar z, \bar u)$ of \eqref{appendix:BZ} exists, we are implicitly assuming $q\geq m$. In fact, if $(\bar z, \bar u)$ is an extreme point of \eqref{appendix:BZ}, then $\bar u$ must be an extreme point of $\{ u: Gu \leq h - H \bar z\}$. This, however, requires that $G$ have at least $m$ rows, thus implying that $q-m \geq 0$ holds.\\

An almost identical result is originally proved in  \cite{bienstock09}. We present such proof here, adapted to our notation and including the $u$ variables, which modifies the final result. Note that this theorem also applies to every \emph{restricted BZ master problem} \eqref{alg:gen:ub}, since the latter is obtained by simply equating sets of variables of \eqref{alg:prob:gen}, thus maintaining the same structure.\\

To prove Lemma \ref{lemma:fractional} we will make use of the \emph{elementary basis} introduced in Section \ref{ssec:bz}, that is, we write
$$\bar{z} = \sum\limits_{i=1}^{k} \lambda_i \bar{v}^i$$

where $\{\lambda_1, \ldots, \lambda_k\} $ are the distinct non-zero values of $\bar{z}$, and each $\bar{v}^i$ is the corresponding indicator vector for $\bar z = \lambda_i$. Without loss of generality we assume $\lambda_1 = 1$, so the fractional values of $\bar z$ are given by $\{\lambda_2, \ldots, \lambda_k\}$.

\begin{lemma}\label{lemma:kernel}
Let $(\bar z, \bar u)$ be an extreme point of \eqref{appendix:BZ}, and decompose $\bar z$ as above. Additionally, denote $\bar A$ the sub-matrix of $A$ corresponding to binding constraints at $(\bar z, \bar u)$. Then $\bar{v}^i \in \text{null}(\bar{A})$ for all $2\leq i \leq k$, and they are linearly independent.
\end{lemma}
\begin{proof}
The $\bar{v}^i$ vectors are clearly linearly independent since they have disjoint support. As for the first claim, consider a precedence constraint $z_i \leq z_j$ $(i,j) \in I$. If such constraint is binding in $(\bar z, \bar u)$ then $\bar z_i = \bar z_j$, which implies 
$$\bar{v}^l_i = 1 \Longleftrightarrow \bar{v}^l_j = 1,\quad l = 1,\ldots, k$$
thus, $\bar{v}^l_i  = \bar{v}^l_j $ $\forall l$. On the other hand, if a constraint $z_i \leq 1$ or $z_i \geq 0$ is binding at $(\bar z, \bar u)$ then clearly $\bar{v}_j^l = 0$ for $l\geq 2$. This proves $\bar{A} \bar{v}^l = 0$ for $2\leq l \leq k$.
\end{proof}

\begin{lemma}\label{lemma:kernelsize}
Let $(\bar z, \bar u)$ and $q$ be as in Lemma \ref{lemma:fractional}. Additionally, let $\bar A$ be the sub-matrix of $A$ corresponding to binding constraints at $(\bar z, \bar u)$. Then $\text{dim}\left( \text{null}(\bar{A}) \right) \leq q - m$.
\end{lemma}
\begin{proof}
Let $\bar H, \bar G$ be the set of rows of $H$ and $G$ corresponding to active constraints in $(\bar z, \bar u)$. Since this is an extreme point, the matrix
$$\left[ \begin{array}{cc} \bar{A} & 0 \\ 
\bar{H} & \bar{G} \end{array} \right] $$
must contain at least $n+m$ linearly independent rows. Suppose $[ \bar H\  \bar G]$ has $\bar q$ linearly independent rows. This implies $\bar{A}$ must have $n+m - \bar q$ linearly independent rows, and in virtue of the rank-nullity theorem,
$$\text{dim}\left( \text{null}(\bar{A}) \right) = n - \text{rank}(\bar{A})  = \bar q - m \leq q - m$$
\end{proof}

Lemma \ref{lemma:fractional} is then obtained as a direct corollary of Lemmas \ref{lemma:kernel} and \ref{lemma:kernelsize}, since these two prove $k-1 \leq q - m$, and $k-1$ is exactly the number of fractional values of $\bar z$.
\end{document}